\documentclass[reqno]{amsart}
\usepackage[utf8]{inputenc}
\usepackage{color}
\usepackage{url}
\usepackage{enumitem}
\usepackage{amstext}
\usepackage{amsthm}
\usepackage{amssymb}
\usepackage{esint}
\usepackage[unicode=true,pdfusetitle,
 bookmarks=true,bookmarksnumbered=false,bookmarksopen=false,
 breaklinks=false,pdfborder={0 0 0},pdfborderstyle={},backref=false,colorlinks=true]
 {hyperref}
\hypersetup{
 pdfborderstyle=,pdfborderstyle={}}

\makeatletter
\numberwithin{equation}{section}
\numberwithin{figure}{section}
\theoremstyle{plain}
\newtheorem{thm}{\protect\theoremname}
\theoremstyle{definition}
\newtheorem{defn}[thm]{\protect\definitionname}
\theoremstyle{plain}
\newtheorem{lem}[thm]{\protect\lemmaname}
\theoremstyle{remark}
\newtheorem{rem}[thm]{\protect\remarkname}
\theoremstyle{plain}
\newtheorem{cor}[thm]{\protect\corollaryname}

\usepackage{pdfsync}
\usepackage{graphics}
\usepackage{epstopdf}
\usepackage[all]{xy}
\usepackage{amscd}
\usepackage{enumitem}
\usepackage{mathtools}
\usepackage{stmaryrd}
\usepackage[utf8]{inputenc}
\pdfoutput=1
\setlist[enumerate]{leftmargin=*,label=(\roman*),align=left}

\makeatletter
\@namedef{subjclassname@2020}{%
  \textup{2020} Mathematics Subject Classification}
\makeatother

\newcommand{\xyR}[1]{ \makeatletter
\xydef@\xymatrixrowsep@{#1} \makeatother} 
\newcommand{\xyC}[1]{ \makeatletter
\xydef@\xymatrixcolsep@{#1} \makeatother} 
\entrymodifiers={++[ ][F]} 

\newcommand{\ra}{\longrightarrow}

\newcommand{\field}[1]{\mathbb{#1}}
\newcommand{\R}{\field{R}} 
\newcommand{\N}{\field{N}} 
\newcommand{\CC}{\field{C}} 

\newcommand{\eps}{\varepsilon} 
\renewcommand{\phi}{\varphi}
\newcommand{\diff}[1]{\ifmmode\mathchoice{\hbox{\rm d}#1}  
 {\hbox{\rm d}#1}  
 {\scalebox{0.75}{$\hbox{\rm d}#1$}}  
 {\scalebox{0.35}{$\hbox{\rm d}#1$}}  
 \fi} 

\newcommand{\abs}[2][\empty]{\ifx#1\empty\left|#2\right|%
\else#1\vert #2 #1\vert\fi}

\newcommand{\Coo}{\mbox{\ensuremath{\mathcal{C}}}^{\infty}} 

\newcommand{\Rtil}{\widetilde \R} 
\newcommand{\Ctil}{\widetilde \CC} 
\newcommand{\Ktil}{\widetilde{\field{K}}} 
\newcommand{\st}[1]{{#1^\circ}} 
\newcommand{\sint}[1]{\langle#1\rangle} 
\newcommand{\Eball}{B^{{\scriptscriptstyle \text{\rm E}}}} 
\newcommand{\nrst}[1]{{#1}^\bullet} 

\newcommand{\frontRise}[2]{\ifmmode\mathchoice{{\vphantom{#1}}^{\scalebox{0.6}{$#2$}}}  
 {{\vphantom{#1}}^{\scalebox{0.56}{$#2$}}}  
 {{\vphantom{#1}}^{\scalebox{0.47}{$#2$}}}  
 {{\vphantom{#1}}^{\scalebox{0.35}{$#2$}}}\fi} 
\newcommand{\RCreal}[1]{\frontRise{\R}{#1}\Rtil}
\newcommand{\RCcomplex}[1]{\frontRise{\CC}{#1}\Ctil}
\newcommand{\RCrealrho}{\RCreal{\rho}}
\newcommand{\RCcomplexrho}{\RCcomplex{\rho}}
\newcommand{\RCKrho}{\frontRise{K}{\rho}\Ktil}
\newcommand{\realpart}{\text{\rm Re}}
\newcommand{\impart}{\text{\rm Im}}
\newcommand{\hyperN}[1]{	\frontRise{\N}{#1}\widetilde{\N}}
\newcommand{\hypNr}{\hyperN{\rho}}
\newcommand{\hypNs}{\hyperN{\sigma}}
\newcommand{\frontRiseDown}[3]{\ifmmode\mathchoice{{\vphantom{#1}}^{\scalebox{0.6}{$#2$}}_{\scalebox{0.6}{$#3$}}}  
 {{\vphantom{#1}}^{\scalebox{0.56}{$#2$}}_{\scalebox{0.56}{$#3$}}}  
 {{\vphantom{#1}}^{\scalebox{0.47}{$#2$}}_{\scalebox{0.47}{$#3$}}}  
 {{\vphantom{#1}}^{\scalebox{0.35}{$#2$}}_{\scalebox{0.35}{$#3$}}}\fi} 

\newcommand{\RCcomplexud}[2]{\frontRiseDown{\CC}{#1}{#2}\Ctil}

\newcommand{\hyperlimfarg}[3]{\mathchoice{\frontRise{\lim}{\raisebox{-0.05em}{$#1\hspace{-0.67em}$}}\lim_{#3\xrightarrow{#2} 0\,}}
	{\frontRise{\lim}{#1\hspace{-0.25em}}\lim_{#3\rightarrow{#2}\,}}
	{\frontRise{\lim}{#1\hspace{-0.25em}}\lim_{#3\rightarrow{#2}\,}}
	{\frontRise{\lim}{#1\hspace{-0.25em}}\lim_{#3\rightarrow{#2}\,}}}
\newcommand{\hyperlimf}[2]{\hyperlimfarg{#1}{#2}{h}}

\newcommand{\hyplimfarg}[3]{\mathchoice{\frontRise{\lim}{\raisebox{-0.05em}{$#1\hspace{-0.67em}$}}\lim_{ #2\rightarrow {#3}\,}}
	{\frontRise{\lim}{#1\hspace{-0.25em}}\lim_{#2 \rightarrow #3\,}}
	{\frontRise{\lim}{#1\hspace{-0.25em}}\lim_{#2 \rightarrow #3\,}}
	{\frontRise{\lim}{#1\hspace{-0.25em}}\lim_{#2 \rightarrow #3\,}}}
\newcommand{\hyplimf}[2]{\hyplimfarg{#1}{#2}{0}}

\newcommand{\gsf}{\frontRise{\mathcal{G}}{\rho}\mathcal{GC}^{\infty}}

\newcommand{\ghf}{\frontRise{\mathcal{G}}{\rho}\mathcal{GH}}

\newcommand{\hypersumarg}[3]{\mathchoice{\frontRise{\sum}{\raisebox{-0.2em}{$#1\hspace{-0.67em}$}}\sum_{#3\in \hyperN{#2}\,}}
{\frontRise{\sum}{#1\hspace{-0.25em}}\sum_{#3\in \hyperN{#2}\,}}
{\frontRise{\sum}{#1\hspace{-0.25em}}\sum_{#3\in \hyperN{#2}\,}}
{\frontRise{\sum}{#1\hspace{-0.25em}}\sum_{#3\in \hyperN{#2}\,}}}

\newcommand{\hypersum}[2]{\hypersumarg{#1}{#2}{n}}

\newcommand{\subzero}{\subseteq_{0}}

\newcommand{\sbpt}[1]{#1_{\text{\rm s}}}

\newcommand{\bigO}{\mathcal{O}}

\makeatother

\providecommand{\corollaryname}{Corollary}
\providecommand{\definitionname}{Definition}
\providecommand{\lemmaname}{Lemma}
\providecommand{\remarkname}{Remark}
\providecommand{\theoremname}{Theorem}

\begin{document}
\title{Dirac delta as a generalized holomorphic function}
\author{Sekar Nugraheni \and Paolo Giordano}
\address{\textsc{Sekar Nugraheni \newline Faculty of Mathematics, University
of Vienna, Austria \newline Faculty of Mathematics and Natural Sciences,
Universitas Gadjah Mada, Indonesia}}
\email{sekar.nugraheni@ugm.ac.id}
\address{\textsc{Paolo Giordano \newline Faculty of Mathematics, University
of Vienna, Austria}}
\email{paolo.giordano@univie.ac.at}
\begin{abstract}
The definition of a non-trivial space of generalized functions of
a complex variable allowing to consider derivatives of continuous
functions is a non-obvious task, e.g.~because of Morera theorem,
because distributional Cauchy-Riemann equations implies holomorphicity
and of course because including Dirac delta seems incompatible with
the identity theorem. Surprisingly, these results can be achieved
if we consider a suitable non-Archimedean extension of the complex
field, i.e.~a ring where infinitesimal and infinite numbers return
to be available. In this first paper, we set the definition of generalized
holomorphic function and prove the extension of several classical
theorems, such as Cauchy-Riemann equations, Goursat, Looman-Menchoff
and Montel theorems, generalized differentiability implies smoothness,
intrinsic embedding of compactly supported distributions, closure
with respect to composition and hence non-linear operations on these
generalized functions. The theory hence addresses several limitations
of Colombeau theory of generalized holomorphic functions. The final
aim of this series of papers is to prove the Cauchy-Kowalevski theorem
including also distributional PDE or singular boundary conditions
and nonlinear operations.
\end{abstract}

\maketitle

\section{Introduction}

The extension of at least some Schwartz distributions from the real
field to the complex plane has been (sometimes informally) attempted
by several authors, because of its uses in quantum physics, \cite{Dav,Bag,Brewster2018,Lin},
in the study of random processes, \cite{Sma}, for physical wavelets,
\cite{Kai}, and in general relativity, \cite{FrPS}. However, these
efforts sometimes present mathematical drawbacks, as already pointed
out, e.g., by \cite{Win}.

From the purely mathematical point of view, the definition of a non-trivial
space of generalized functions of a complex variable that allows one
to consider derivatives of continuous functions is a non-obvious task,
because its solution must sidestep several impossibility theorems.
For example, if we want that these generalized functions embed ordinary
continuous maps defined on a domain $D\subseteq\mathbb{C}$ and, at
the same time, satisfy the Cauchy integral theorem, then these continuous
functions would also be path-independent and, from Morera's theorem,
they would actually be holomorphic functions, see e.g.~\cite{Zal}.
Likewise, if we want that these generalized functions satisfy the
Cauchy-Riemann equations (CRE), even with respect to distributional
derivatives, then necessarily the embedded continuous ones will actually
be, once again, ordinary holomorphic functions, see \cite{GrMo}.
Paraphrasing the language of \cite{Hor90} (see also \cite{KeGi24}),
we could say that the simplest solution of the problem to have derivatives
of continuous functions of a complex variable satisfying the Cauchy
theorem or the Cauchy-Riemann equations is the sheaf of holomorphic
functions itself, and a larger space seems not possible. Once again:
it seems impossible to extend the Dirac delta distribution from $\R$
to $\mathbb{C}$ and, at the same time, to prove a general identity
theorem for these \emph{generalized holomorphic functions} (GHF).
The problem is also related to nonlinear operations on spaces of distributions:
in fact, one of the peculiar property of holomorphic functions is
their expandability in Taylor series, and hence the converge of the
Cauchy product of these series seems a natural consequence, if we
think to generalize this expandability to GHF; however, nonlinear
operations, even involving only Heaviside function and Dirac delta
(assuming that this expandability can also be proved at least for
these GHF), collide with well-known impossibility theorems about nonlinear
operations on distributions, see e.g.~\cite{Colombeau1984b,Grosser2001}
and references therein.

On the other hand, a full theory of GHF would open the possibility
to generalize the Cauchy-Kowalevski theorem including also singular
(e.g.~using suitable Schwartz distributions) PDE or singular boundary
conditions and nonlinear (generalized holomorphic) operations. This
is the final aim of our study, which starts with the present paper.

It is rather surprising that all these impossibility results can be
avoided if we consider \emph{non-Archimedean extensions} $\RCcomplexrho\supseteq\mathbb{C}$
and $\RCrealrho\supseteq\R$ of the complex and real fields, i.e.~rings
$\RCcomplexrho$ and $\RCrealrho$ where infinitesimal and infinite
numbers return to be available. In fact, for our GHF (defined on an
open subset of $\RCcomplexrho$ and valued in this same ring) we can
prove the Cauchy integral theorem, the CRE, Morera's theorem, expandability
in Taylor hyperseries (i.e.~series extended over natural infinite
numbers, see Def.~\ref{def:hypernaturalNumbers} below), closure
with respect to composition, and hence nonlinear operations, and the
generalization of several other results of ordinary holomorphic functions.
On the other hand, in essentially every non-Archimedean theory, see
e.g.~\cite{Ehr} and references therein, the topology of interest
always see the set of all the infinitesimals as a clopen set, and
hence the corresponding non-Archimedean ring is disconnected (and
all its intervals as well). This property allows one to prove only
a weak form of the identity theorem for GHF and hence make it possible
to extend the Dirac delta from $\RCrealrho^{2}$ to $\RCcomplexrho$.
Any complex primitive of this extended Dirac delta can also be regarded
as an extension of the Heaviside function and, using the closure of
GHF with respect to composition and nonlinear operations, among GHF
we can include examples such as, e.g., $\delta^{a}\cdot H^{b}\circ\delta^{c}\cdot H^{d}$
for any $a$, $b$, $c$, $d\in\N=\{0,1,2,3,\ldots\}$. Moreover,
in our solution, every GHF is a set-theoretical map defined on a subset
of $\RCcomplexrho$ and valued in the same ring, and this solves the
informal use of generalized functions as one can find in physics and
engineer, \cite{BrIsGi}. Using a classical smoothing method, every
continuous map $f$ defined on a domain $\Omega\subseteq\mathbb{C}$
and satisfying the distributional CRE can be embedded as a GHF, but
if we assume to lose all the infinitesimal information, formally if
the regularized values are ordinary complex numbers for all $z\in\Omega$,
then the previous result of \cite{GrMo} about distributional CRE
implies that necessarily $f$ is a standard holomorphic function,
see Thm.~\ref{thm:emb2}. Informally, we can hence state that without
a suitable language of infinitesimal and infinite numbers, a theory
of GHF with meaningful examples and results is impossible, but in
an appropriate non-Archimedean setting, this theory is possible and
full of reach examples and results.

In this first paper, we introduce the rings $\RCrealrho$ and $\RCcomplexrho$,
the definition of GHF, and prove its first properties, we prove that
generalized $\RCcomplexrho$-differentiability implies smoothness,
the CRE, versions of Goursat, Looman-Menchoff, and Montel theorems,
and close with a list of meaningful examples, including all compactly
supported distributions, hence the aforementioned $\delta$ and its
relation with the one dimensional Dirac delta, and all the locally
integrable functions satisfying the distributional CRE. We also summarize
why our approach solves several technical problems of Colombeau setting:
intrinsic embedding of distributions preserving all $\delta$ derivatives,
closure with respect to composition, possibility to have generalized
functions defined in infinitesimal or infinite sets (hence solutions
of differential equations which are impossible for Colombeau generalized
functions (CGF)), and better Fourier transform applicable also to
non-tempered distributions.

In subsequent papers \cite{NuGi24a,NuGi24b}, already almost finished,
we will present the extension of classical properties of path integration
to GHF, and the expansion using hyper power series, i.e.~power series
where the summation is extended over infinite natural numbers of the
ring $\RCrealrho$.

The paper is self-contained in the sense that it contains all the
statements required for the proofs we are going to present. If proofs
of preliminaries are omitted, we clearly give references to where
they can be found. Therefore, to understand this paper, only a basic
knowledge of distribution theory is needed.

\subsection{\protect\label{subsec:The-Ring-ofRC}The Ring of Robinson-Colombeau
Numbers}

A sufficiently general rigorous mathematical theory of generalized
functions of a complex variable has already been developed within
Colombeau theory of generalized functions (see e.g.~\cite{Colombeau1984,Aragona2005,Oberguggenberger2007,Vernaeve2008,Aragona2012,Juriaans2020}
and references therein). Since the beginning, in that framework it
was natural to define a holomorphic generalized function using CRE,
see e.g.~\cite{Colombeau1984,Vernaeve2008,Aragona2005,Aragona2012}.
Indeed, the notion of partial derivatives (of any order) of a CGF
was already clear and fully compatible with distributional derivative;
therefore, using pointwise evaluation of CGF, one can define a complex
CGF of a complex variable starting from its real and imaginary parts
and asking that they satisfy the CRE. Even if an important results
such as the CRE and the existence of \emph{all} the derivatives are
directly taken in the definition, this initial approach probably appeared
more natural than considering the limit of the incremental ratio in
a \emph{ring} of scalars (see Sec.~\ref{subsec:The-Ring-ofRC} below)
with a topology usually managed through a non-Archimedean ultrametric.
On the other hand, already in \cite{Aragona2005}, thanks to the density
of invertible elements in the ring of scalars, a more general notion
of differentiability in the Colombeau setting, both for the complex
and the real case, is considered using a classical Newton quotient.
This allows \cite{Aragona2005} to show that all CGF are also differentiable
in this more classical way, but missed to extend the important classical
result that from a suitable notion of \emph{generalized} complex differentiability
the existence of all \emph{generalized} derivatives of greater order
follows, see also \cite{Vernaeve2008}.

In this article we want to define a GHF by using some kind of limit
of the incremental ratio, like for ordinary holomorphic functions,
without already starting from the CRE and also showing that first
order generalized complex differentiability implies the existence
of all greater derivatives. After our proof of the CRE for GHF, see
Thm.~\ref{thm:CRE}, we will be able to show that our approach to
GHF is more general than the classical Colombeau definition, and this
allows us to include examples such as the Dirac delta and the closure
with respect to composition.

In this section, we introduce the non-Archimedean ring of (real and
complex) scalars. For more details and proofs about the basic notions
introduced here, the reader can refer e.g.~to \cite{Giordano2021,Colombeau1984b,Giordano2015}.
As we will see better below, in order to accomplish our definition
of GHF, we introduce Colombeau generalized numbers by considering
an arbitrary asymptotic scale instead of the usual net $(\eps)$ used
in Colombeau theory (see also \cite{Marti1999} for a more general
notion of scale, and \cite{Giordano2016} and references therein for
a comparison).
\begin{defn}
\label{def:RCring}Let $\rho=(\rho_{\eps}):(0,1]\ra(0,1]=:I$ be a
net such that $(\rho_{\eps})\rightarrow0$ as $\eps\rightarrow0^{+}$
(in the following, such a net will be called a \textit{gauge} and
all the asymptotic relation will be for $\eps\to0^{+}$), then
\begin{enumerate}
\item We say that a net $(x_{\eps})\in\R^{I}$ is \textit{$\rho$-moderate},
and we write $(x_{\eps})\in\R_{\rho}$ if $\exists N\in\N:\ x_{\eps}=\bigO(\rho_{\eps}^{-N})$.
\item Let $(x_{\eps})$, $(y_{\eps})\in\R^{I}$, then we say that $(x_{\eps})\sim_{\rho}(y_{\eps})$,
and we read it saying that $(x_{\eps}-y_{\eps})$ is $\rho$-\emph{negligible},
if $\forall N\in\N:\ |x_{\eps}-y_{\eps}|=\bigO(\rho_{\eps}^{N})$.
This is a congruence relation on the ring $\R_{\rho}$ of moderate
nets with respect to pointwise operations, and we can hence define
\[
\RCrealrho:=\R_{\rho}{\large /\sim_{\rho},}
\]
which we call \textit{Robinson-Colombeau ring of generalized numbers}\textit{\emph{.}}
The corresponding equivalence classes are simply denoted by $x=[x_{\eps}]$
or $x=[x_{\eps}]_{\rho}$ in case we have to underscore the dependence
from the gauge $\rho$. In particular, $\diff\rho:=[\rho_{\eps}]\in\RCrealrho$.
\item If $\mathcal{P}(\eps)$ is a property of $\eps\in I$, we use the
notation $\forall^{0}\eps:\ \mathcal{P}(\eps)$ to denote $\exists\eps_{0}\in I\,\forall\eps\in(0,\eps_{0}]:\ \mathcal{P}(\eps)$.
We can read $\forall^{0}\eps$ as: ``\emph{for $\eps$ small''}.
\item Let $x$, $y\in\RCrealrho$. We write $x\leq y$ if for all representative
$[x_{\eps}]=x$, there exists $[y_{\eps}]=y$ such that $\forall^{0}\eps:\ x_{\eps}\leq y_{\eps}$.
\item We denote by $\RCrealrho_{>0}$ the set of positive invertible generalized
numbers. In general, we write $x<y$ to say that $x\le y$ and $x-y$
is invertible.
\item A generalized complex number can be written as $z=x+iy\in\RCcomplexrho$,
where $x$, $y\in\RCrealrho$ and $i$ is the imaginary unit. Note
that $\R\subseteq\RCrealrho$ and $\CC\subseteq\RCcomplexrho$ are
embedded using constant nets.
\end{enumerate}
\end{defn}

In $\RCrealrho$, we can easily define infinitesimal and infinite
numbers, which coincide with infinitesimal and infinite representatives
(this result does not hold in nonstandard analysis, see e.g.~\cite{CKKR}):
\begin{defn}
\label{def:nonArchNumbs}Let $z\in\RCcomplexrho$ be a generalized
number. Then
\begin{enumerate}
\item $z$ is \emph{infinitesimal} if $|z|\le r$ for all $r\in\R_{>0}$.
If $z=[z_{\eps}]$, this is equivalent to $\lim_{\eps\to0^{+}}z_{\eps}=0$.
We write $z\approx w$ if $z-w$ is infinitesimal, and $D_{\infty}:=\left\{ h\in\RCcomplexrho\mid h\approx0\right\} $
for the set of all infinitesimals in $\RCcomplexrho$.
\item $z$ is \emph{infinite} if $|z|\ge r$ for all $r\in\R_{>0}$. If
$z=[z_{\eps}]$, this is equivalent to $\lim_{\eps\to0^{+}}\left|z_{\eps}\right|=+\infty$.
\item $z$ is \emph{finite} if $|z|\le r$ for some $r\in\R_{>0}$.
\item $z$ is \emph{near-standard} if for some representative (and hence
for all) $[z_{\eps}]=z$, we have $\exists\lim_{\eps\to0^{+}}z_{\eps}=:\st{z}\in\CC$,
which is called the \emph{standard part} of $z$. If $U\subseteq\RCcomplexrho$,
the set of near-standard points in $U$ is $\nrst{U}:=\left\{ z\in U\mid\exists\st{z}\in U\right\} $.
\end{enumerate}
\end{defn}

\noindent For example, we have that $\diff{\rho}^{n}\in\RCreal{\rho}$,
$n\in\N_{>0}$, is an invertible infinitesimal, whose reciprocal is
$\diff{\rho}^{-n}=[\rho_{\eps}^{-n}]$, which is necessarily a positive
infinite number. Of course, in the ring $\RCreal{\rho}$ there exist
generalized numbers which are not in any of the three classes of Def.~\ref{def:nonArchNumbs},
like e.g.~$x_{\eps}=\frac{1}{\eps}\sin\left(\frac{1}{\eps}\right)$.
It is possible to prove that $\RCrealrho$ is the simplest (co-universal)
quotient ring having a prescribed element $\diff\rho$ and where every
representative of $0=[z_{\eps}]$ is an infinitesimal net: $\lim_{\eps\to0^{+}}z_{\eps}=0$,
see \cite{KeGi24}.

On $\RCrealrho^{n}$, we consider the natural extension of the Euclidean
norm, i.e. $|[x_{\eps}]|:=[|x_{\eps}|]\in\RCrealrho$. Even if this
generalized norm takes value in $\RCrealrho$, it shares essential
properties with classical norms, like the triangle inequality and
absolute homogeneity. It is therefore natural to consider on $\RCrealrho^{n}$
the topology generated by balls $B_{r}(x):=\left\{ y\in\RCrealrho\mid|x-y|<r\right\} $,
$r\in\RCrealrho_{>0}$, which is called \textit{sharp topology}. Note
that $\diff\rho\in\RCrealrho_{>0}$, and therefore we can also have
balls $B_{\diff\rho^{n}}(x)$ of infinitesimal radius. Similarly,
the absolute value (modulus) of a generalized complex number $z=[z_{\eps}]=[x_{\eps}+iy_{\eps}]=x+iy\in\RCcomplexrho$
is defined by $|z|:=|\left[z_{\eps}\right]|=\left[\left(x_{\eps}^{2}+y_{\eps}^{2}\right)^{\frac{1}{2}}\right]$
and takes value in $\RCrealrho$. Therefore, the topology of the set
of generalized complex numbers is the same as the sharp topology of
$\RCrealrho^{2}$.

The following result is useful in dealing with positive and invertible
generalized numbers, see \cite{Grosser2001,Giordano2021}.
\begin{lem}
\label{lem:mayer} Let $x\in\RCrealrho$. Then the following are equivalent:
\begin{enumerate}
\item \label{enu:positiveInvertible}$x$ is invertible and $x\ge0$, i.e.~$x>0$.
\item \label{enu:strictlyPositive}For each representative $(x_{\eps})\in\R_{\rho}$
of $x$ we have $\forall^{0}\eps:\ x_{\eps}>0$.
\item \label{enu:greater-i_epsTom}For each representative $(x_{\eps})\in\R_{\rho}$
of $x$ we have $\exists m\in\N\,\forall^{0}\eps:\ x_{\eps}>\rho_{\eps}^{m}$,
i.e.~$x>\diff\rho^{m}$.
\item \label{enu:There-exists-a}There exists a representative $(x_{\eps})\in\R_{\rho}$
of $x$ such that $\exists m\in\N\,\forall^{0}\eps:\ x_{\eps}>\rho_{\eps}^{m}$.
\end{enumerate}
\end{lem}

\noindent A consequence of this lemma is that the sharp topology is
generated by all the infinitesimal balls of the type $B_{\diff\rho^{q}}(z)$
for any $q\in\N$ and $z\in\RCcomplexrho$, and that $x<y$ is equivalent
to $x_{\eps}<y_{\eps}$ for $\eps$ small and for all representatives
$[x_{\eps}]=x$ and $[y_{\eps}]=y$.
\begin{lem}
\label{lem:Invertible-generalized-complex}Every $z\in\RCcomplexrho$
is invertible if and only if $|z|$ is invertible. Moreover, invertible
generalized complex numbers are dense in $\RCcomplexrho$.
\end{lem}

\begin{proof}
If $z\cdot w-1=[z_{\eps}\cdot w_{\eps}-1]=0$, then $z_{\eps}\cdot w_{\eps}-1\to0$
for $\eps\to0^{+}$. Therefore, by contradiction, $z_{\eps}\ne0$
for $\eps$ small, and hence $|z_{\eps}|>0$. By Lem.~\ref{lem:mayer},
the absolute value $|z|$ is invertible. Conversely, assume that $|z|$
is invertible. Since $|z|\geq0$, by Lem.~\ref{lem:mayer}, for every
representative $(z_{\eps})\in\CC_{\rho}$ of $z$ we have $\forall^{0}\eps:\ |z_{\eps}|>0$
and hence $z$ is invertible with inverse $[z_{\eps}^{-1}]$. For
the second part of the statement, we have to prove that for all $r\in\RCrealrho_{>0}$
and $z\in\RCcomplexrho$ there exists an invertible $z^{*}\in B_{r}(z)$.
Let $z=[z_{\eps}]$ and $r=[r_{\eps}]$ be any representatives, we
set
\[
z_{\eps}^{*}:=\begin{cases}
\frac{r_{\eps}}{2} & \text{if }z_{\eps}=0\\
z_{\eps} & \text{otherwise}
\end{cases},
\]
Then, once again by Lem.~\ref{lem:mayer}, we have that $z^{*}$
is invertible and $z^{*}\in B_{r}(z)$.
\end{proof}
A natural way to obtain particular subsets of $\frontRise{\mathbb{K}}{\rho}\mathbb{\widetilde{K}}^{n}$,
$\mathbb{K}\in\{\R,\CC\}$, is by using a net $(A_{\eps})$ of subset
$A_{\eps}\subseteq\mathbb{K}$.
\begin{defn}
\label{def:intStr}Let $(A_{\eps})$ be a net of subsets of $\mathbb{K}$,
then
\begin{enumerate}
\item A set of the type 
\[
[A_{\eps}]:=\left\{ [x_{\eps}]\in\frontRise{\mathbb{K}}{\rho}\mathbb{\widetilde{K}}^{n}\mid\forall^{0}\eps:\ x_{\eps}\in A_{\eps}\right\} 
\]
is called \emph{internal set}. It is possible to show that $[A_{\eps}]$
is sharply closed and $[A_{\eps}]=[\text{cl}(A_{\eps})]$, where $\text{cl}(S)$
is the closure of $S\subseteq\mathbb{K}^{n}$, see e.g.~\cite{Oberguggenberger2008,Giordano2021}.
\item A set of the type 
\[
\left\langle A_{\eps}\right\rangle :=\left\{ x\in\frontRise{\mathbb{K}}{\rho}\mathbb{\widetilde{K}}^{n}\mid\forall[x_{\eps}]=x\,\forall^{0}\eps:\ x_{\eps}\in A_{\eps}\right\} 
\]
is called \emph{strongly internal set}. It is possible to show that
$\sint{A_{\eps}}$ is sharply open and $\sint{A_{\eps}}=\sint{\text{int}(A_{\eps})}$,
where $\text{int}(S)$ is the interior of $S\subseteq\mathbb{K}^{n}$,
see e.g.~\cite{Giordano2015}.
\end{enumerate}
\end{defn}

It is quite useful to intuitively think at $\eps\to0^{+}$ as a time
variable, and hence $x=[x_{\eps}]$ as a \emph{dynamical number }(static
numbers are ordinary reals)\emph{ with a fuzzy halo of $\rho$-negligible
amplitude around $(x_{\eps})$}. A similar intuitive interpretation
can be considered for (strongly) internal sets. A thorough investigation
of internal sets can be found in \cite{Oberguggenberger2008,Aragona2012};
see \cite{Giordano2015} for strongly internal sets.

For example, it is not hard to show that the open and closed balls
of center $c=[c_{\eps}]\in\RCcomplexrho$ and radius $r=[r_{\eps}]>0$
can be written as 
\begin{equation}
B_{r}(c)=\left\langle \Eball_{r_{\eps}}(c_{\eps})\right\rangle ,\quad\overline{B_{r}(c)}=\left[\Eball_{r_{\eps}}(c_{\eps})\right],\label{eq:ball}
\end{equation}
where $\Eball_{r_{\eps}}(c_{\eps}):=\left\{ z\in\CC\mid|z-c_{\eps}|<r_{\eps}\right\} $
denotes an ordinary Euclidean ball in $\CC$. Moreover, by contradiction,
it is easy to prove that
\begin{equation}
\overline{B_{r}(c)}\subseteq\sint{A_{\eps}}\ \Rightarrow\ \forall^{0}\eps:\ \Eball_{r_{\eps}}(c_{\eps})\subseteq A_{\eps},\label{eq:closedBallStr}
\end{equation}
see e.g.~\cite[Lem.~11]{Giordano2021}.

Before continuing with the formal mathematics, we want to explain
why considering an arbitrary gauge $\rho$ is important. In fact,
any good theory of GHF has to be well linked with a good notion of
power series. However, in the ring $\Rtil:=\RCreal{(\eps)}$ of Colombeau
generalized numbers we have that $(x_{n})_{n\in\N}\in\Rtil^{\N}$
is a Cauchy sequence if and only if $\lim_{n\to+\infty}\left|x_{n+1}-x_{n}\right|=0$
(in the sharp topology; see \cite{Mukhammadiev2021,Tiwari2022}).
As a consequence, a series of Colombeau generalized numbers $a_{n}\in\Rtil$
\begin{equation}
\sum_{n=0}^{\infty}a_{n}\text{ converges}\ \iff\ a_{n}\to0\text{ (in the sharp topology)}.\label{eq:convNonArch}
\end{equation}
Actually, this is a well-known property of every ultrametric space,
cf., e.g., \cite{Koblitz1996}. For example, $\sum_{n=1}^{\infty}\frac{1}{n^{2}}\in\Rtil$
converges in the sharp topology if and only if $\frac{1}{n^{2}}\to0$
in the same topology, for $n\to+\infty$, $n\in\N_{>0}$. But the
sharp topology on $\Rtil$ necessarily has to deal with balls having
infinitesimal radius $r\in\Rtil$ (because generalized functions can
have infinite derivatives and are continuous in this topology, see
also \cite{Giordano2021}), and thus $\frac{1}{n^{2}}\not\rightarrow0$
if $n\rightarrow+\infty$, $n\in\N_{>0}$, because we never have $\R_{>0}\ni\frac{1}{n^{2}}<r$
if $r$ is infinitesimal. Similarly, one can easily prove that if
the exponential series $\sum_{n=0}^{+\infty}\frac{x^{n}}{n!}$ converges,
then necessarily $x\approx0$ must be an infinitesimal number because
if $\frac{x^{n}}{n!}\to0$ in the sharp topology of $\Rtil$, then
$|x_{\eps}|\le n!\eps^{1/n}$ for all $n\in\N$ sufficiently large.

Intuitively, the only way to have $\frac{1}{n^{2}}<r\approx0$ is
to take for $n\in\Rtil$ an infinite number, in this case an infinite
natural number. On the other hand, we intuitively would like to have
$\frac{1}{\log n}\to0$, but we have $\frac{1}{\log n_{\eps}}<\eps^{q}$
if and only if $n_{\eps}>e^{\eps^{-q}}$ and the net $\left(e^{\eps^{-q}}\right)_{\eps}$
is not $(\eps)$-moderate. In order to settle this problem, it is
hence important to generalize the role of the net $(\eps)$ as used
in Colombeau theory, into a more general gauge $\rho=(\rho_{\eps})\rightarrow0$,
and hence to generalize $\Rtil$ into $\RCrealrho$.

We can consider the aforementioned set of infinite natural numbers,
called \emph{hypernatural numbers,} and the set of $\rho$-moderate
nets of natural numbers in the following:
\begin{defn}
\label{def:hypernaturalNumbers}We set
\begin{enumerate}
\item $\hypNr:=\left\{ [n_{\eps}]\in\RCrealrho\mid n_{\eps}\in\N\;\forall\eps\right\} $.
\item $\N_{\rho}:=\left\{ \left(n_{\eps}\right)\in\R_{\rho}\mid n_{\eps}\in\N\;\forall\eps\right\} $
\end{enumerate}
\end{defn}

\noindent Therefore, $n\in\hypNr$ if and only if there exists $(x_{\eps})\in\R_{\rho}$
such that $n=\left[\text{int}\left(\left|x_{\eps}\right|\right)\right]$,
where $\text{int}\left(-\right)$ is the integer part function. Note
that $\text{int}\left(-\right)$ is not well-defined on $\hypNr$:
In fact, if $x=1=\left[1-\rho_{\eps}^{1/\eps}\right]=\left[1+\rho_{\eps}^{1/\eps}\right]$,
then $\text{int}\left(1-\rho_{\eps}^{1/\eps}\right)=0$ whereas $\text{int}\left(1+\rho_{\eps}^{1/\eps}\right)=1$,
for $\eps$ sufficiently small. However, the nearest integer function
can be correctly defined as follows (see e.g.~\cite{Mukhammadiev2021}):
\begin{defn}
\label{def:niF}The \textit{nearest integer function }$\text{ni}\left(-\right)$
is defined by:
\begin{enumerate}
\item $\text{ni}:\hypNr\longrightarrow\N_{\rho}$
\item If $[x_{\eps}]\in\hypNr$ and $\text{ni}\left(\left[x_{\eps}\right]\right)=\left(n_{\eps}\right)$
then $\forall^{0}\eps:$ $n_{\eps}=\left\lfloor x_{\eps}+\frac{1}{2}\right\rfloor $,
where $\left\lfloor -\right\rfloor $ is the floor function.
\end{enumerate}
In other words, if $x\in\hypNr$, then $x=\left[\text{ni}\left(x\right)_{\eps}\right]$
and $\text{ni}\left(x\right)_{\eps}\in\N$ for all $\eps$.
\end{defn}

To glimpse the necessity of studying $\hypNr$, it suffices to note
that we have $\frac{1}{\log n}\rightarrow0$ in $\RCreal{\sigma}$
as $n\to+\infty$ for $n\in\hypNs$, but only for a suitable gauge
$\sigma$ (depending on $\rho$, e.g.~$\sigma_{\eps}:=\exp(-\rho_{\eps}^{-1/\rho_{\eps}})$),
whereas this limit does not exist if $\sigma=\rho$ (cf.~\cite[Example~27]{Mukhammadiev2021}).
This is also related to the former problem of series of generalized
numbers $a_{n}\in\RCrealrho$ because instead of ordinary series,
we have better results summing over all $n\in\hypNs$: e.g.~we have
$\hypersum{\rho}{\sigma}\frac{z^{n}}{n!}=e^{z}$ for all $z\in\RCcomplexrho$
where the exponential is moderate, i.e.~if $|z_{\eps}|\le\log\left(\rho_{\eps}^{-R}\right)$
for some $R\in\N$ (clearly, this includes all finite numbers $z$),
see \cite{NuGi24b,Tiwari2023}. Intuitively, the smaller is the second
gauge $\sigma$, the greater are the infinite numbers we can consider
with it (i.e.~represented by $\sigma$-moderate nets), and hence
the greater is the number of summands in summations of the form $\hypersum{\rho}{\sigma}a_{n}\in\RCrealrho$.
This kind of summations are called \emph{hyperseries}, their theory
has been developed in \cite{Tiwari2022,Tiwari2023}, and we will use
them in \cite{NuGi24b} to show that any GHF is expandable in Taylor
hyperseries.

We now introduce the notion of \emph{generalized smooth function},
because the real and imaginary part of every GHF always belong to
this class.

\subsection{\protect\label{subsec:Generalized-Smooth-Functions}Generalized Smooth
Functions}

Generalized smooth functions (GSF) are the simplest way to deal with
a very large class of generalized functions and singular problems,
by working directly with all their $\rho$-moderate smooth regularizations,
in fact they are a universal way to have set-theoretical functions
defined on generalized numbers and having arbitrary derivatives (see
\cite{KeGi24}). GSF are close to the historically original conception
of generalized function, \cite{Dir26,Lau89,Kat-Tal12}: in essence,
the idea of authors such as Dirac, Cauchy, Poisson, Kirchhoff, Helmholtz,
Kelvin and Heaviside (who informally worked with ``numbers'' which
also comprise infinitesimals and infinite scalars) was to view generalized
functions as certain types of smooth set-theoretical maps obtained
from ordinary smooth maps by introducing a dependence on suitable
infinitesimal or infinite parameters. For example, the density of
a Cauchy-Lorentz distribution with an infinitesimal scale parameter
was used by Cauchy to obtain classical properties which nowadays are
attributed to the Dirac delta, \cite{Kat-Tal12}. More generally,
in the GSF approach, generalized functions are seen as suitable maps
defined on, and attaining values in, the non-Archimedean ring of scalars
$\RCrealrho$. The calculus of GSF is closely related to classical
analysis sharing several properties of ordinary smooth functions.
On the other hand, GSF include all Colombeau generalized functions
and hence also all Schwartz distributions \cite{Giordano2015,Grosser2001,Giordano2021}.
They allow nonlinear operations on generalized functions and unrestricted
composition \cite{Giordano2015,Giordano2021}. For GSF, we can prove
a number of analogues of theorems of classical analysis for generalized
functions: e.g., mean value theorem, intermediate value theorem, extreme
value theorem, Taylor’s theorems, local and global inverse function
theorems, integrals via primitives, and multidimensional integrals
\cite{Giordano2018,Giordano2017,Giordano2021}. We can develop calculus
of variations and optimal control for generalized functions, with
applications e.g.~in collision mechanics, singular optics, quantum
mechanics and general relativity, see \cite{LLG,Gastao2022}. We have
new existence results for nonlinear singular ODE and PDE (e.g.~a
Picard-Lindelöf theorem for PDE), \cite{Lu-Gi16,GiLu}, and with the
notion of \emph{hyperfinite Fourier transform} we can consider the
Fourier transform of any GSF, without restriction to tempered type,
\cite{MTG}. GSF with their particular sheaf property define a Grothendieck
topos, \cite{Giordano2021}.
\begin{defn}
\label{def:GSF}Let $X\subseteq\RCrealrho^{n}$ and $Y\subseteq\RCrealrho^{d}$.
We say that $f:X\ra Y$ is a GSF ($f\in\gsf(X,Y)$), if
\begin{enumerate}
\item $f:X\ra Y$ is a map.
\item There exists a net $(f_{\eps})\in\mathcal{C}^{\infty}(\Omega_{\eps},\R^{d})$
such that $X\subseteq\sint{\Omega_{\eps}}$ and for all $[x_{\eps}]\in X$
(i.e.~for all representatives $(x_{\eps})$):
\begin{enumerate}[label=\alph*)]
\item $f(x)=[f_{\eps}(x_{\eps})]$
\item $\forall\alpha\in\N^{n}:\ \left(\partial^{\alpha}f_{\eps}(x_{\eps})\right)$
is $\rho$-moderate.
\end{enumerate}
\end{enumerate}
\end{defn}

In Thm.~\ref{thm:functionofderivative}, we will prove that real
and imaginary part of any GHF are GSF.

\subsection{The language of subpoints.}

The following simple language allow us to simplify some proofs using
steps recalling the classical real field $\R$, particularly in dealing
with a negation of an equality or of an inequality, or even of an
arbitrary property in $\RCrealrho$. This notion is studied in \cite{Mukhammadiev2021}.
\begin{defn}
\label{def:subpoint}~
\begin{enumerate}
\item For subsets $J$, $K\subseteq I$ we write $K\subseteq_{0}J$ if $0$
ia an accumulation point of $K$ and $K\subseteq J$ (we read it as:
$K$ \textit{is co-final in }$J$). Note that for any $J\subseteq_{0}I$,
the constructions introduced so far in Def. \ref{def:RCring} can
be repeated using nets $(x_{\eps})_{\eps\in J}$. We indicate the
resulting ring with the symbol $\RCrealrho|_{J}$. More generally,
no peculiar property of $I=(0,1]$ will ever be used in the following,
and hence all the presented results can be easily generalized considering
any other directed set.
\item If $K\subseteq_{0}J$, $x\in\RCrealrho|_{J}$ and $x'\in\RCrealrho|_{K}$,
then $x'$ is called a \textit{subpoint} of $x$, denoted as $x'\subseteq x$,
if there exist representatives $(x_{\eps})_{\eps\in J}$, $(x_{\eps}')_{\eps\in K}$
of $x$, $x'$ such that $x_{\eps}'=x_{\eps}$ for all $\eps\in K$.
In this case, we write $x'=x|_{K}$, $\text{dom}(x'):=K$, and the
restriction $(-)|_{K}:\RCrealrho\ra\RCrealrho|_{K}$ is a well defined
operation. In general, for $X\subseteq\RCrealrho$ we set $X|_{J}:=\left\{ x|_{J}\in\RCrealrho|_{J}:x\in X\right\} $.
\end{enumerate}
\end{defn}

In the next definition, we now consider binary relations that hold
only \textit{on subpoints}.
\begin{defn}
Let $x$, $y\in\RCrealrho$, $L\subseteq_{0}I$, then we say:
\begin{enumerate}
\item $x<_{L}y$ if $x|_{L}<y|_{L}$, the latter inequality has to be meant
in the ordered ring $\RCrealrho|_{L}$. We read $x<_{L}y$ as ``$x$
is less than $y$ on $L$''.
\item $x\sbpt{<}y$ if $\exists L\subseteq_{0}I:x<_{L}y$. We read $x\sbpt{<}y$
as ``$x$ is less than $y$ on subpoints''.
\end{enumerate}
Analogously, we can define other relations holding only on subpoints
such as e.g.: $\sbpt{\in}$, $\sbpt{\le}$, $\sbpt{=}$, $\sbpt{\subseteq}$,
etc.
\end{defn}

For example, we have
\begin{align}
x<y & \Longleftrightarrow\forall L\subseteq_{0}I:x<_{L}y\;\text{and}\nonumber \\
x\neq y & \Longleftrightarrow x\sbpt{>}y\text{ or }x\sbpt{<}y.\label{eq:orderedFieldSub}
\end{align}
Moreover, if $\mathcal{P}\{x_{\eps}\}$ is an arbitrary property of
$x_{\eps}$, then
\[
\neg\left(\forall^{0}\eps:\mathcal{P}\{x_{\eps}\}\right)\Longleftrightarrow\exists L\subseteq_{0}I\;\forall\eps\in L:\neg\mathcal{P}\{x_{\eps}\}.
\]
As one can guess from \eqref{eq:orderedFieldSub} (recall that $x\sbpt{>}0$
implies that $x$ is invertible on subpoints), this language allows
one to deal with good substitute of the field and of the total order
properties, see \cite{Mukhammadiev2021} for a deeper study.

\section{\protect\label{sec:Generalized-Holomorphic-Function}Generalized
Holomorphic Functions}

\subsection{\protect\label{subsec:Basic-and-Little-oh}Basic (locally Lipschitz)
functions, limits and weak little-oh}

Before defining the notion of GHF, we need to reformulate in our non-Archimedean
setting some more basic notions: what is the class of functions we
want to consider, their limits and continuity in the sharp topology,
and Landau little-oh relation. These are the concepts that allow us
to get a beautiful theory of GHF.

Considering the theory of GSF, we still want to have generalized holomorphic
function of the form $f(z)=\left[f_{\eps}(z_{\eps})\right]\in\RCcomplexrho$
but without assuming the CRE for $f_{\eps}$ directly in the definition
like in Colombeau theory; we can think e.g.~$(f_{\eps})$ as obtained
by a suitable holomorphic regularization of a compactly supported
distribution, see Sec.~\ref{sec:Examples}.

We therefore start by considering the notion of basic function as
follows:
\begin{defn}
\label{def:basicF}Let $U\subseteq\RCcomplexrho$. We say that $f:U\ra\RCcomplexrho$
is a \textit{basic function} (or a \textit{$\rho$-basic function}
in case we have to underscore the dependence from the gauge $\rho$),
if
\begin{enumerate}
\item \label{enu:basicF1}$f:U\ra\RCcomplexrho$ is a map;
\item \label{enu:basicF2}There exists a net $(f_{\eps})$ such that $f_{\eps}:U_{\eps}\ra\CC$,
$U\subseteq\sint{U_{\eps}}$, $U_{\eps}\subseteq\CC$, and for all
$z=[z_{\eps}]\in U$
\begin{equation}
f(z)=\left[f_{\eps}\left(z_{\eps}\right)\right],\label{eq:basicF}
\end{equation}
and in this case, we say that $f$ \emph{is defined by} the net $(f_{\eps})$.
\item If $U=\sint{\Omega}$, the \emph{standard part of }$f$ is $\st{f}:\left\{ z\in\Omega\mid\exists\st{f(z)}\right\} \mapsto\st{f(z)}\in\CC$,
so that $\st{f}(z)=\lim_{\eps\to0^{+}}f_{\eps}(z)\in\CC$ for all
$z\in\Omega$ where such a limit exists.
\end{enumerate}
\noindent Moreover, we say that the basic function $f$ is ($\rho$-\emph{)locally
Lipschitz} if for all $z\in U$, there exist constants $L\in\RCrealrho_{>0}$
and $r\in\RCrealrho_{>0}$ such that 
\begin{equation}
|f(x)-f(y)|\leq L\cdot|x-y|\quad\forall x,y\in B_{r}(z).\label{eq:lipcond}
\end{equation}
We also refer to $f$ as $L$-\emph{Lipschitz function} on the sharply
open ball $B_{r}(z)$ if \eqref{eq:lipcond} holds, and we finally
say that $f$ is \emph{Lipschitz at $z$} if \eqref{eq:lipcond} holds
for some $r\in\RCrealrho_{>0}$ and some $L\in\RCrealrho_{>0}$.
\end{defn}

\noindent A similar concept was first introduced in \cite{Garetto2011}
to define basic linear maps. Note also that equality \eqref{eq:basicF}
implies that if $(z_{\eps})\sim_{\rho}(z'_{\eps})$, then $(f_{\eps}(z_{\eps}))\sim_{\rho}(f_{\eps}(z'_{\eps}))$,
which is a strong limitation on the maps $f_{\eps}$: e.g.~$f_{\eps}(z)=|z|$
defines a basic function, but $f_{\eps}(z)=1$ if $z\ne0$ and $f_{\eps}(0)=0$
does not. Therefore, basic functions are particular cases of set theoretical
maps of the type $f:U\ra\RCcomplexrho$ because they need to have
suitable defining nets $(f_{\eps})$ satisfying \ref{enu:basicF2}
above.

Clearly, we expect that our GHF are locally Lipschitz. We therefore
start studying locally Lipschitz maps and proving that being locally
Lipschitz can also be formulated $\eps$-wise.
\begin{thm}
\label{thm:Lipschitzcond}Let $U_{0}\subseteq\RCcomplexrho$ be a
sharp neighbourhood of $0$. If $R=[R_{\eps}(-)]:U_{0}\ra\RCcomplexrho$
is a basic function, then $R$ is Lipschitz at $0$ if and only if
there exists $Q\in\R_{>0}$ and $(L_{\eps})\in\R_{\rho}$ such that
$R_{\eps}$ is $L_{\eps}$-Lipschitz (hence continuous) over $\Eball_{\rho_{\eps}^{Q}}(0)\subseteq\CC$
for $\eps$ small.
\end{thm}

\begin{proof}
$\Rightarrow$: Property \eqref{eq:lipcond} yields the existence
of some $L=[L_{\eps}]\in\RCrealrho_{>0}$ and $r=[r_{\eps}]\in\RCrealrho_{>0}$
such that $R$ is $L$-Lipschitz on the sharply open ball $B_{r}(0)$.
From \eqref{eq:ball} and Lem.~\ref{lem:mayer}, there exists $Q\in\N$
such that $\forall^{0}\eps:\ r_{\eps}>\rho_{\eps}^{Q}$ and hence
$R$ is $L$-Lipschitz on $\left\langle \Eball_{\rho_{\eps}^{Q}}(0)\right\rangle $.
By contradiction, assume that there exist a strictly decreasing sequence
$(\eps_{n})_{n\in\N}\rightarrow0^{+}$ and $h_{\eps_{n}}$, $k_{\eps_{n}}\in\Eball_{\rho_{\eps_{n}}^{Q}}(0)$
such that
\begin{equation}
|R_{\eps_{n}}(h_{\eps_{n}})-R_{\eps_{n}}(k_{\eps_{n}})|>L_{\eps_{n}}|h_{\eps_{n}}-k_{\eps_{n}}|\quad\forall n\in\N.\label{eq:KContr}
\end{equation}
Defining $h_{\eps}:=h_{\eps_{n}}$ and $k_{\eps}:=k_{\eps_{n}}$ for
$\eps\in(\eps_{n+1},\eps_{n}]$, and taking any $h_{\eps}$, $k_{\eps}\in\Eball_{\rho_{\eps}^{Q}}(0)$
otherwise, we have that $h:=[h_{\eps}]$, $k:=[k_{\eps}]\in B_{r}(0)$.
Setting $K:=\left\{ \eps_{n}\mid n\in\N\right\} \subseteq_{0}I$,
we have $|h-k|>_{K}0$ from \eqref{eq:KContr}, and
\[
L|h-k|<_{K}|R(h)-R(k)|\leq L\cdot|h-k|\Longrightarrow L\cdot|h-k|<_{K}L\cdot|h-k|.
\]
Therefore, $L<_{K}L$ for some $K\subseteq_{0}I$, which is a contradiction.

\noindent$\Leftarrow$: For $\eps$ small and $h$, $k\in\Eball_{\rho_{\eps}^{Q}}(0)$,
we have
\[
|R_{\eps}(h)-R_{\eps}(k)|\leq L_{\eps}|h-k|.
\]
For all $x=[x_{\eps}]$, $y=[y_{\eps}]\in\left\langle \Eball_{\rho_{\eps}^{Q}}(0)\right\rangle =B_{\diff\rho^{Q}}(0)$,
we have $x_{\eps}$, $y_{\eps}\in\Eball_{\rho_{\eps}^{Q}}(0)$ and
\[
|R_{\eps}(x_{\eps})-R_{\eps}(y_{\eps})|\leq L_{\eps}\cdot|x_{\eps}-y_{\eps}|,
\]
for small $\eps$. Hence, $R$ is locally Lipschitz for $L:=[L_{\eps}]\in\RCrealrho$.
\end{proof}
It is clear that any basic locally Lipschitz function $f:U\ra\RCcomplexrho$
is \textit{sharply continuous on }$U$, i.e.~for all $z_{0}\in U$,\textit{
\begin{equation}
\forall q\in\N\,\exists H\in\RCrealrho_{>0}\,\forall h\in U:|h-z_{0}|<H\implies\left|f(h)-f(z_{0})\right|<\diff\rho^{q}.\label{eq:contF}
\end{equation}
}

\noindent In the next definition, we introduce the notion of \emph{limit}
of a basic functions.
\begin{defn}
\label{def:rho-limit}Let $U_{0}\subseteq\RCcomplexrho$ be a sharp
neighbourhood of $z_{0}$. Let $R:U_{0}\ra\RCcomplexrho$ be a basic
function. We say that the ($\rho$-)\emph{limit} of $R$, as $h$
tends to $z_{0}$ in the ($\rho$-)sharp topology is $\lambda\in\RCcomplexrho$,
and we write $\lim_{h\to z_{0}}R(h)=\lambda$ or $\hyperlimfarg{\rho}{z_{0}}{h}R(h)=\lambda$
in case we have to underscore the dependence from the gauge $\rho$,
if the following property holds: 
\begin{equation}
\forall q\in\N\,\exists H\in\RCreal\rho_{>0}\,\forall h\in U_{0}:\ 0<|h-z_{0}|<H\Rightarrow|R(h)-\lambda|<\diff{\rho}^{q}.\label{eq:defhyperlimit}
\end{equation}
\end{defn}

\begin{rem}
\label{rem:rholimit}In the assumptions of Def.~\ref{def:rho-limit},
let $k\in\RCrealrho_{>0}$ and $N\in\N$, then the following are equivalent:
\begin{enumerate}
\item $\lambda\in\RCcomplexrho$ is the $\rho$-limit of $R$ as $h$ tends
to $z_{0}$.
\item $\forall\eta\in\RCreal\rho_{>0}\,\exists H\in\RCreal\rho_{>0}\,\forall h\in U_{0}:\ 0<|h-z_{0}|<H\Rightarrow|R(h)-\lambda|<\eta$.
\item \label{enu:rholimit3}For all $\rho$-sharply open sets $U\subseteq\RCcomplexrho$
then there exists a $\rho$-sharply open set $V\subseteq\RCcomplexrho$
containing $z_{0}$ such that $\forall h\in V\cap U_{0}:R(h)\in U$.
\item $\forall q\in\N\,\exists H\in\RCreal\rho_{>0}\,\forall h\in U_{0}:\ 0<|h-z_{0}|<H\Rightarrow|R(h)-\lambda|<k\cdot\diff{\rho}^{q}$.
\item $\forall q\in\N\,\exists H\in\RCreal\rho_{>0}\,\forall h\in U_{0}:\ 0<|h-z_{0}|<H\Rightarrow|R(h)-\lambda|<\diff{\rho}^{q-N}$.
\end{enumerate}
\end{rem}

\noindent Directly by the inequality $|\lambda_{1}-\lambda_{2}|\leq|\lambda_{1}-R(h)|+|\lambda_{2}-R(h)|\leq2\diff\rho^{q+1}<\diff\rho^{q}$
and by Lem.~\ref{lem:Invertible-generalized-complex} (or by using
that the sharp topology on $\RCcomplexrho$ is Hausdorff) it follows
that there exists at most one limit. Moreover, since any basic locally
Lipschitz function $R:U_{0}\ra\RCcomplexrho$ is sharply continuous
on $U_{0}$, in this case we have 
\begin{equation}
\lim_{h\to z_{0}}R(h)=R(z_{0}).\label{eq:limitcontinuous}
\end{equation}

\begin{lem}
\label{lem:R(0)=00003D0}Let $U_{0}\subseteq\RCcomplexrho$ be a sharp
neighbourhood of $0$. Let $R:U_{0}\ra\RCcomplexrho$ be a basic locally
Lipschitz function. Then $R(0)=0$ if and only if there exists a net
of functions $(R_{\eps}(-))$ that defines $R$ such that $R_{\eps}(0)=0$
for small $\eps$.
\end{lem}

\begin{proof}
If $R(0)=0$ and the net $(\bar{R}_{\eps})$ defines $R$, then $R_{\eps}(h):=\bar{R}_{\eps}(h)-\bar{R}_{\eps}(0)$
satisfies the claim. The opposite implication is trivial.
\end{proof}
In the following corollary, we use Thm.~\ref{thm:Lipschitzcond}
to get an $\eps$-wise characterization of limits.
\begin{cor}
\label{cor:epslimit}Let $U_{0}\subseteq\RCcomplexrho$ be a sharp
neighbourhood of $0$. Let $R:U_{0}\ra\RCcomplexrho$ be a $\rho$-basic
function which is Lipschitz at $0$. Then the following are equivalent:
\begin{enumerate}
\item $\hyperlimf{\rho}{0}R(h)=0$.
\item \label{enu:epslimit(2)}There exist $Q\in\R_{>0}$ and a net of functions
$(R_{\eps}(-))$ that defines $R$, it is $L_{\eps}$-Lipschitz over
$\Eball_{\rho_{\eps}^{Q}}(0)\subseteq\CC$, and $\lim_{h\rightarrow0}R_{\eps}(h)=0$
for small $\eps$.
\end{enumerate}
\end{cor}

\medskip{}

A natural definition of $\RCcomplexrho$-differentiable function $f:U\ra\RCcomplexrho$
at $z_{0}\in U$ has to be linked to a linearization property of the
form
\begin{equation}
f(z_{0}+h)=f(z_{0})+h\cdot m+o(h)\text{ as }h\to0,\label{eq:linearization}
\end{equation}
for some $m\in\RCcomplexrho$. In pursuing this idea, we have to first
define the notion of Landau little-oh in a ring such as $\RCcomplexrho$.
\begin{defn}
\label{def:weakLittleoh}Let $U_{0}\subseteq\RCcomplexrho$ be a sharp
neighbourhood of $0$ and $f_{1}$, $f_{2}:U_{0}\ra\RCcomplexrho$
be maps. Then we say that 
\[
f_{1}(h)=\bar{o}(f_{2}(h))\:\text{as}\:h\to0,
\]
if there exists a map $r:U_{0}\ra\RCcomplexrho$ such that
\begin{enumerate}
\item \label{enu:weakLittleoh1}$\hyplimf{\rho}{h}r(h)=0$.
\item \label{enu:weakLittleoh2}$\forall h\in U_{0}:\ f_{1}(h)=f_{2}(h)r(h)$.
\end{enumerate}
We call this relation \emph{weak little-oh}.

\end{defn}

\begin{rem}
\label{rem:littleoh}~
\begin{enumerate}[label=\alph*)]
\item If $f_{2}(h)$ is invertible for all sufficiently small invertible
$h$, and both $r$, $\bar{r}$ satisfy Def.~\ref{def:weakLittleoh},
then $r(h)=\bar{r}(h)$ for these $h$.
\item If $f_{2}(h)=h$, then Def.~\ref{def:weakLittleoh} implies $\hyplimf{\rho}{h}\frac{f_{1}(h)}{h}=0$,
because in the definition of limit \eqref{eq:defhyperlimit} we always
take $0<|h|$, i.e.~$h$ invertible. This and \eqref{eq:linearization}
yield a comparison with the approach followed by \cite{Aragona2005}.
\end{enumerate}
\end{rem}

\noindent Here are some general rules that are easily satisfied by
the definition of weak little-oh:
\begin{lem}
\label{lem:littleoh}In the assumptions of Def.~\ref{def:weakLittleoh},
we have
\begin{enumerate}
\item \label{enu:lemmalittleoh1}If $f(h)=\bar{o}(h)$ and $g(h)=\bar{o}(h)$
then $f(h)+g(h)=\bar{o}(h)$ and $f(h)g(h)=\bar{o}(h^{2})$.
\item \label{enu:lemmalittleoh2}If $f(h)=\bar{o}(h)$ and $c\in\RCcomplexrho$,
then $cf(h)=\bar{o}(h)$.
\end{enumerate}
\end{lem}

\subsection{\protect\label{subsec:Complex-differentiable-function}Complex generalized
differentiable function}

Having the previous notion of weak little-oh relation, it would now
be natural to use the linearization property \eqref{eq:linearization}
to define a pointwise complex differentiable function. However, considering
the following motivations, we prefer to introduce a stronger notion
of little-oh:
\begin{enumerate}[label=\arabic*)]
\item \label{enu:i}Def.~\ref{def:weakLittleoh} of weak little-oh do
not use the notion of basic function at all, and indeed it includes
the classical example given by $i(z):=1$, if $z$ is infinitesimal
and $i(z):=0$ otherwise. This function $i$ cannot be a GHF on any
neighbourhood $B_{r}(0)$ defined by a standard radius $r\in\R_{>0}$,
because, otherwise, its real and imaginary parts would be GSF (see
\ref{enu:GHF-GSF} of Thm.~\ref{thm:functionofderivative} below),
which is not the case because the intermediate value theorem does
not hold for them, see \cite[Cor.~48]{Giordano2021} and Sec.~\ref{subsec:Generalized-Smooth-Functions}
above). In particular, $i(z)$ cannot be globally defined $i(z)=[i_{\eps}(z_{\eps})]$
by a net of ordinary holomorphic functions $(i_{\eps})$. On the other
hand, the function $i$ seems well-behaving, at least if we restrict
$i$ to infinitesimal neighbourhoods and standard points, i.e.~on
balls of the type $B_{r}(c)$, where $c\in\CC$ and $r\approx0$,
e.g.~$r=\diff\rho^{Q}$.
\item \label{enu:i-series}We would then have the drawback that the function
$i$ does not coincide with its Taylor series in its entire set of
convergence, even if the latter is the trivial (constant) series.
Therefore, in general, if we define GHF using \eqref{eq:linearization}
and the weak little-oh, we cannot prove that if $f$ is a GHF, its
Taylor series converges to $f$ in the entire convergence set. All
this is related to the total disconnectedness property of $\RCcomplexrho$,
due to the fact that the set of infinitesimals is a clopen set.
\item \label{enu:globalNets}We therefore follow the idea to ask a relation
of the type \eqref{eq:linearization}, but remaining in the class
of basic functions for all the maps appearing in \eqref{eq:linearization},
$f(z_{0}+h)-f(z_{0})-h\cdot m$ and $o(h)=o(z_{0},h)$, and for all
the variables $z_{0}$ and $h$. In other words, the previous function
$i$ satisfies a property of the form ``for any standard point $z_{0}\in\CC$
there exists a net of functions $(i_{\eps})$ defining $i$ in an
infinitesimal neighbourhood of $z_{0}$ and such that the linearization
property \eqref{eq:linearization} holds'', whereas we want to restrict
to a stronger definition of the form ``there exists a net of functions
$(f_{\eps})$ globally defining $f$, and at the point $z_{0}\in\RCcomplexrho$
a linearization property of the form \eqref{eq:linearization} holds''.
\item \label{enu:strongLittle-oh}Finally, we also want to prove that $f'(z_{0}):=m$
is still a basic function defined by a net $(f'_{\eps})$ of ordinary
complex differentiable functions because this would ensure that our
class of GHF is not too large but remains strictly tied to regularizations
of GF using ordinary holomorphic ones, see Sec.~\ref{sec:Examples}.
Starting only from \eqref{eq:linearization}, this desired property
does not hold: in fact, since $[\rho_{\eps}^{1/\eps}]=0\in\RCcomplexrho$,
we could have $f_{\eps}(z_{0\eps}+h)=f_{\eps}(z_{0\eps})+h\cdot m_{\eps}+h\cdot o_{\eps}(z_{0\eps},h)+\rho_{\eps}^{1/\eps}\cdot n_{\eps}(z_{0\eps},h)$
for $\eps$ and $h$ small, even if $n_{\eps}(z_{0\eps},h)$ is not
holomorphic in its variables. For this reason, we ask that \eqref{eq:linearization}
holds also at the $\eps$-level and without any other non-regular
additive remainder.
\item \label{enu:betterThanCol}As we already mentioned in Sec.~\ref{subsec:The-Ring-ofRC},
one of the aim of this work is to prove, in this setting of GF, that
a condition of first order differentiability implies both the CRE
and the existence of all greater derivatives.
\end{enumerate}
We denote by $\bigO(\Omega)$ the space of holomorphic functions on
$\Omega$, where $\Omega$ is an open subset of $\CC$.
\begin{defn}
\label{def:complexdiff}Let $\emptyset\ne V\subseteq U\subseteq\RCcomplexrho$
and $f:U\ra\RCcomplexrho$ be a basic function. Then $f$ is said
to be $\RCcomplexrho$\textit{-differentiable on $V$} if there exist
$U_{0}\subseteq\RCcomplexrho$ and nets $f_{\eps}:U_{\eps}\ra\CC$,
$r_{\eps}:V_{\eps}\times U_{0\eps}\ra\CC$ such that:
\begin{enumerate}
\item \label{enu:neighborhood}$\forall z\in V\,\exists s\in\RCrealrho_{>0}:\ z+B_{s}(0)\subseteq z+U_{0}\subseteq U$,
\item \label{enu:f_eps}$(f_{\eps})$ defines $f$,
\item \label{enu:r_eps}$(r_{\eps})$ defines a basic function $r:V\times U_{0}\ra\RCcomplexrho$.
\end{enumerate}
For all $z\in V$, we have:
\begin{enumerate}[resume]
\item \label{enu:limit-r}$\hyplimf{\rho}{h}r(z,h)=0$.
\item \label{enu:epsLinDiff}There exists $[m_{\eps}]=m\in\RCcomplexrho$
such that for all representatives $[z_{\eps}]=z\in V$ and all $[h_{\eps}]\in U_{0}$,
we have
\begin{align}
\forall^{0}\eps:\ f_{\eps}(z_{\eps}+h_{\eps}) & =f_{\eps}(z_{\eps})+h_{\eps}\cdot m_{\eps}+h_{\eps}\cdot r_{\eps}(z_{\eps},h_{\eps}).\label{eq:epswisecomplexdiff}\\
r_{\eps}(z_{\eps},0) & =0=\lim_{h\rightarrow0}{r_{\eps}(z_{\eps},h)}.\label{eq:limitofresidue}
\end{align}
\end{enumerate}
If conditions \ref{enu:neighborhood} - \ref{enu:epsLinDiff} hold,
we simply write
\begin{equation}
f(z+h)=f(z)+h\cdot m+o(h)\text{ as }h\to0\text{ in }z\in V.\label{eq:generalizedComplexDiff1}
\end{equation}
The relation $o(h)$ is called \emph{strong little-oh}, or simply
\emph{little-oh}.

\noindent Finally, we say that:
\begin{enumerate}[label=\arabic*)]
\item $f$ is $\RCcomplexrho$\emph{-differentiable} \emph{at} $z_{0}$
if it is $\RCcomplexrho$-differentiable on $V=\{z_{0}\}$.
\item $f$ is a \textit{generalized holomorphic function }\textit{\emph{(GHF)}}
at $z_{0}$ if there exists a neighbourhood $V\supseteq B_{r}(z_{0})$
of $z_{0}$ such that it is $\RCcomplexrho$-differentiable on $V$
and the functions $f_{\eps}$ of Def.~\ref{def:complexdiff} are
all holomorphic: $f_{\eps}\in\bigO(\Eball_{r_{\eps}}(z_{0\eps}))\ \forall\eps$,
where $r=[r_{\eps}]\in\RCrealrho_{>0}$ and $z_{0}=[z_{0\eps}]$.
\item \label{enu:defGHF}$f$ is a \textit{GHF on a set} $V\subseteq U$,
if $V$ is sharply open and $f$ is $\RCcomplexrho$-differentiable
on $V$ and the functions $f_{\eps}$ of Def.~\ref{def:complexdiff}
are all holomorphic: $f_{\eps}\in\bigO(U_{\eps})\ \forall\eps$.
\item We write $f\in\ghf(W,Y)$ if $f:W\ra Y$ and $f$ is a GHF on $W$,
or simply $f\in\ghf(W)$ if $Y=\RCcomplexrho$.
\end{enumerate}
\ 

\end{defn}

\begin{rem}
\label{rem:complexdiff}~
\begin{enumerate}[label=\alph*)]
\item Note explicitly that Def.~\ref{def:complexdiff} is of the form
$\exists(f_{\eps})\,\exists(r_{\eps})\,\forall z\in V$ and not $\forall z\in V\,\exists(f_{\eps})\,\exists(r_{\eps})$,
as we already mentioned in \ref{enu:globalNets} at the beginning
of the present section.
\item Property \ref{enu:neighborhood} of Def.~\ref{def:complexdiff} allows
us to consider both $f(z+h)$ and $r(z,h)$ for all $z\in V$ and
$h\in B_{s}(0)$. If we set $U_{0}(z):=\left\{ h\in U_{0}\mid z+h\in U\right\} $
for each $z\in V$, we have $U_{0}(z)\supseteq B_{s}(0)$ and $U_{0}=\bigcup_{z\in V}U_{0}(z)$.
Moreover, \ref{enu:neighborhood} also implies that $V$ is contained
in the sharply interior of $U$.
\item \label{enu:epsdiff}By contradiction, it is not hard to prove that
\eqref{eq:epswisecomplexdiff} implies
\begin{equation}
\forall^{0}\eps\,\forall h\in\Eball_{s_{\eps}}(0):\ f_{\eps}(z_{\eps}+h)=f_{\eps}(z_{\eps})+h\cdot m_{\eps}+h\cdot r_{\eps}(z_{\eps},h),\label{eq:forall-h}
\end{equation}
where $z+\overline{B}_{s}(0)\subseteq z+U_{0}\subseteq U$. In fact,
assume that for some $J\subseteq_{0}I$, we can find $h_{\eps}\in\Eball_{s_{\eps}}(0)$
for $\eps\in J$ and such that 
\begin{equation}
f_{\eps}(z_{\eps}+h_{\eps})\neq f_{\eps}(z_{\eps})+h_{\eps}\cdot m_{\eps}+h_{\eps}\cdot r_{\eps}(z_{\eps},h_{\eps}).\label{eq:notforall-h}
\end{equation}
Defining $h_{\eps}:=0$ for $\eps\in I\setminus J$, we have $[h_{\eps}]\in[\Eball_{s_{\eps}}(0)]=\overline{B}_{s}(0)\subseteq U_{0}$,
and hence \eqref{eq:forall-h} contradicts \eqref{eq:epswisecomplexdiff}.
From \eqref{eq:forall-h} and \eqref{eq:limitofresidue} of Def.~\ref{def:complexdiff},
it follows that for $\eps$ small, each function $f_{\eps}$ is differentiable
at $z_{\eps}\in U_{\eps}$.
\item Properties \ref{enu:limit-r} and \eqref{eq:epswisecomplexdiff} of
Def.~\ref{def:complexdiff} imply 
\begin{equation}
f(z+h)=f(z)+h\cdot m+h\cdot r(z,h),\label{eq:generalizedComplexDiff2}
\end{equation}
for all $z\in V$ and all $h\in U_{0}(z)$, i.e.~a weak little-oh
relation of the form \eqref{eq:linearization}. Clearly, all the present
construction of GHF is based on the classical theory of holomorphic
functions, and indeed we are using a net $(f_{\eps})\in\bigO(U_{\eps})$
in Def.~\ref{def:complexdiff}.\ref{enu:defGHF} of GHF; however,
the notion of differentiability \eqref{eq:generalizedComplexDiff2}
or \eqref{eq:generalizedComplexDiff1} for the \emph{generalized}
function $f$ is stated only as a first order condition. Up to this
point, is hence an open problem both whether the CRE holds for a GHF
$f$ and how to even define higher order derivatives.
\item \label{enu:sheafHol}It is important to note that conditions \eqref{eq:epswisecomplexdiff}
and \eqref{eq:generalizedComplexDiff2}, express a condition only
on generalized points $z=[z_{\eps}]\in V$, and not on standard points
$z\in U_{\eps}$, which instead is considered by the properties of
the type $f_{\eps}\in\bigO(\Eball_{r_{\eps}}(z_{0\eps}))$ or $f_{\eps}\in\bigO(U_{\eps})$.
For example, we can even have that $V=\{[z_{\eps}]\}$ is given by
a single infinite point.
\item We recall that in the case $V=\{z_{0}\}$, a basic function $r:\{z_{0}\}\times U_{0}\ra\RCcomplexrho$
is defined by $r_{\eps}:V_{\eps}\times U_{0\eps}\ra\CC$ if $\{z_{0}\}\times U_{0}\subseteq\sint{V_{\eps}\times U_{0\eps}}$,
and $\left[r_{\eps}(\bar{z}_{0\eps},\bar{h}_{\eps})\right]=r(z_{0},h)=\left[r_{\eps}(z_{0\eps},h_{\eps})\right]$
for all representatives $[\bar{z}_{0\eps}]=z_{0}=[z_{0\eps}]$ and
all $h=[h_{\eps}]=[\bar{h}_{\eps}]\in U_{0}$, see Def.~\ref{def:basicF}.
Note that from $0\in U_{0}$ it also follows $(z_{0\eps},0)\in V_{\eps}\times U_{0\eps}$
for $\eps$ small.
\item Let $q\in\N$. Since $\lim_{h\rightarrow0}r(z,h)=0$, then there exists
$H\in\RCrealrho_{>0}$ such that $\forall h\in U_{0}(z):0<|h|<H$,
we have $|r(z,h)|<\diff\rho^{q}$. Hence, 
\[
\left|\frac{f(z+h)-f(z)}{h}-m\right|=|r(z,h)|<\diff\rho^{q},\quad\forall h\in U_{0}(z):0<|h|<H.
\]
Recall Lem.~\ref{lem:Invertible-generalized-complex} about the density
of invertible elements, and that $|h|>0$ means that $|h|$ is invertible.
In the other words, we have 
\[
m=\lim_{h\rightarrow0}\frac{f(z+h)-f(z)}{h}.
\]
Therefore, $m$ is unique and we can define 
\[
f'(z)=m=\lim_{h\rightarrow0}\frac{f(z+h)-f(z)}{h}
\]
which is called \textit{the derivative of $f$} \emph{at }$z\in V$.
Here, we can only conclude that $f'(z)$ is a number dependent on
$z$. To prove that $f'(-)$ is a basic function, we need the following:
\end{enumerate}
\end{rem}

\begin{thm}
\label{thm:functionofderivative}In the assumptions of Def.~\ref{def:complexdiff},
let $f$ be a \emph{GHF} on $U$ and $[z_{\eps}]=z\in V$. Then, we
have:
\begin{enumerate}
\item \label{enu:f_eps_locHol}$\exists Q\in\R_{>0}\,\forall^{0}\eps:\ f_{\eps}|_{B_{\rho_{\eps}^{Q}}(z_{\eps})}$
is a holomorphic function and $f'_{\eps}(z_{\eps})=m_{\eps}$ for
$\eps$ small, so that $f'(z)=[f'_{\eps}(z_{\eps})]$.
\item \label{enu:epsDerAreModerate}$\exists Q,R\in\R_{>0}\,\forall^{0}\eps\,\forall n\in\N:\ \left|\frac{f_{\eps}^{(n)}(z_{\eps})}{n!}\right|\le\rho_{\eps}^{-nQ-R}$.
In particular, all $\eps$-wise derivatives are $\rho$-moderate:
$\forall n\in\N:\ \left(f_{\eps}^{(n)}(z_{\eps})\right)\in\CC_{\rho}$.
\item \label{enu:r_epsLocLip}The function $r_{\eps}$ in Def.~\ref{def:complexdiff}
satisfies 
\[
\exists[a_{\eps}]\in\RCrealrho_{>0}\,\forall^{0}\eps\,\forall h\in\Eball_{a_{\eps}}(z_{\eps}):\ r_{\eps}(z_{\eps},h)=\sum_{n=2}^{\infty}\frac{f^{(n)}(z_{\eps})}{n!}h^{n-1}.
\]
Therefore, it is locally Lipschitz in the variable $h$, and hence
the $\rho$-limit Def.~\ref{def:complexdiff}.\ref{enu:limit-r}
always follows by the $\eps$-wise limit Def.~\ref{def:complexdiff}.\ref{enu:epsLinDiff}.
\item \label{enu:r_epsIntForm}The function $r_{\eps}$ in Def.~\ref{def:complexdiff}
can also be written as 
\begin{equation}
\exists[s_{\eps}]\in\RCrealrho_{>0}\,\forall^{0}\eps\,\forall h\in\Eball_{s_{\eps}}(z_{\eps}):\ r_{\eps}(z_{\eps},h)=\int_{0}^{1}f'_{\eps}(z_{\eps}+ht)\,\diff t-f'_{\eps}(z_{\eps}).\label{eq:intForm}
\end{equation}
\item \label{enu:GHF-GSF}Real and imaginary parts of $f$ are GSF defined
on $\left\{ (x,y)\in\RCrealrho^{2}\mid x+iy\in U\right\} $ and hence
$f$ is locally Lipschitz.
\item \label{enu:derivativefunction2}If $U$ is sharply open and $f\in\ghf(U)$,
the derivative of $f$, $f'\in\ghf(U)$. Therefore, recursively defining
the derivatives of $f$, we have that $f^{(n)}\in\ghf(U)$ for all
$n\in\N$.
\item \label{enu:rLocLip}The function $r:=[r_{\eps}(-,-)]$ is locally
Lipschitz in both variables $z$, $h$.
\end{enumerate}
\end{thm}

\begin{proof}
\noindent\ref{enu:f_eps_locHol}: In fact $z=[z_{\eps}]\in V\subseteq U$
and $U\subseteq\sint{U_{\eps}}$ because $(f_{\eps})$ defines $f$.
Therefore, there exists $Q\in\N$ such that $\overline{B_{\diff\rho^{Q}}(z)}\subseteq\sint{U_{\eps}}$
because $\sint{U_{\eps}}$ is a sharply open set (see Def.~\ref{def:intStr}),
and hence $\Eball_{\rho_{\eps}^{Q}}(z_{\eps})\subseteq U_{\eps}$
for $\eps$ small by \eqref{eq:closedBallStr}. This yields the conclusion
because $f_{\eps}\in\mathcal{O}(U_{\eps})$. Finally, \eqref{eq:forall-h}
and \eqref{eq:limitofresidue} imply $f_{\eps}'(z_{\eps})=m_{\eps}$
for small $\eps$, so that $f'(z)=m=[m_{\eps}]=[f'_{\eps}(z_{\eps})]$.

\noindent\ref{enu:epsDerAreModerate}: From Def.~\ref{def:complexdiff}.\ref{enu:neighborhood}
we can find $a=[a_{\eps}]$, $b=[b_{\eps}]\in\RCrealrho_{>0}$ with
$a<b$ and such that $B_{a}(z)\subseteq\overline{B_{b}(z)}\subseteq U$.
As we proved in \ref{enu:f_eps_locHol}, we have $f_{\eps}\in\mathcal{O}(\Eball_{b_{\eps}}(z_{\eps}))$
for $\eps$ small, and hence for these $\eps$ Cauchy's formula yields
\begin{equation}
f_{\eps}^{(n)}(z_{\eps})=\frac{n!}{2\pi i}\int_{\gamma}\frac{f_{\eps}(z)}{(z-z_{\eps})^{n+1}}\diff z,\label{eq:Cauchyformula}
\end{equation}
where $\gamma$ is the counterclockwise oriented circle forming the
boundary of $\Eball_{a_{\eps}}(z_{\eps})$. Therefore, $|f_{\eps}^{(n)}(z_{\eps})|\le n!a_{\eps}^{-n}\sup_{z\in\gamma}|f_{\eps}(z)|$.
Since $f_{\eps}$ is continuous and $\gamma\subseteq\CC$ is compact,
we have $\sup_{z\in\gamma}|f_{\eps}(z)|=|f_{\eps}(\bar{z}_{\eps})|$
for some $\bar{z}_{\eps}\in\gamma$, and hence $\bar{z}:=[\bar{z}_{\eps}]\in U$.
Thereby $\left(\left|f_{\eps}(\bar{z}_{\eps})\right|\right)\in\R_{\rho}$
because $f:U\ra\RCcomplexrho$ is a basic function, and this proves
claim \ref{enu:epsDerAreModerate}, where we can take $Q\in\R_{>0}$
such that $a\ge\diff\rho^{Q}$, hence depending only on $a$, and
$R\in\R_{>0}$ such that $|f_{\eps}(\bar{z}_{\eps})|\le\rho_{\eps}^{-R}$.
Moreover, if $\eps_{0}\in I$ satisfies both $a_{\eps}\ge\rho_{\eps}^{Q}$
and $|f_{\eps}(\bar{z}_{\eps})|\le\rho_{\eps}^{-R}$ for all $\eps\le\eps_{0}$,
then for all $n\in\N$ we also have $\left|f_{\eps}^{(n)}(z_{\eps})\right|\le n!\rho_{\eps}^{-nQ-R}$.

\noindent\ref{enu:r_epsLocLip}: Taking $a$, $b\in\RCrealrho_{>0}$
as above, we have that the Taylor formula
\begin{equation}
f_{\eps}(z)=\sum_{n=0}^{\infty}\frac{f_{\eps}^{(n)}(z_{\eps})}{n!}(z-z_{\eps})^{n}\label{eq:Taylorformula}
\end{equation}
pointwise converges for $z\in\Eball_{a_{\eps}}(z_{\eps})$. The conclusion
now follows from \eqref{eq:forall-h}.

\noindent\ref{enu:r_epsIntForm}: From \eqref{eq:forall-h}, for
$h\in\Eball_{s_{\eps}}(z_{\eps})\setminus\{0\}$ it suffices to multiply
by $h$ both sides of \eqref{eq:intForm} and integrate. For $h=0$,
the equality follows from $r_{\eps}(z_{\eps},0)=0$.

\noindent\ref{enu:GHF-GSF}: This follows from \ref{enu:f_eps_locHol},
\ref{enu:epsDerAreModerate} and Def.~\ref{def:GSF}.

\noindent\ref{enu:derivativefunction2}: Set $U'_{0}:=\bigcup\left\{ B_{a}(0)\mid a>0,\ \exists b>a\,\exists z\in U:\ \overline{B_{b}(z)}\subseteq U\right\} $.
For each $z\in U$, from Def.~\ref{def:complexdiff}.\ref{enu:neighborhood}
(for $V=U$), it follows the existence of $a$, $b\in\RCrealrho_{>0}$
such that $a<b$ and $z+B_{a}(0)\subseteq\overline{B_{b}(z)}\subseteq z+U_{0}\subseteq U$.
Thereby, $B_{a}(0)\subseteq U'_{0}$ and hence $z+B_{a}(0)\subseteq z+U'_{0}\subseteq U$.
This shows property \ref{enu:neighborhood} for $U'_{0}$. The map
$z=[z_{\eps}]\in U\mapsto[f'_{\eps}(z_{\eps})]\in\RCcomplexrho$ is
basic because of property \ref{enu:epsDerAreModerate} and because
if $\left(z_{\eps}\right)\simeq_{\rho}\left(\bar{z}_{\eps}\right)$,
then $\left|f'_{\eps}(z_{\eps})-f'_{\eps}(\bar{z}_{\eps})\right|\le|z_{\eps}-\bar{z}_{\eps}|\left|f''_{\eps}(w_{\eps})\right|\simeq_{\rho}0$
for some $w_{\eps}$ in the segment $[z_{\eps},\bar{z}_{\eps}]$.
Using \ref{enu:r_epsIntForm}, we can define $r'_{\eps}(z,h):=\int_{0}^{1}f''_{\eps}(z+ht)\,\diff t-f''_{\eps}(z)$
for all $z\in U_{\eps}$ and all $h$ such that $[z,z+h]\subseteq U_{\eps}$.
Once again from \ref{enu:epsDerAreModerate} it follows that $r_{\eps}$
defines a basic function of the type $U\times U'_{0}\ra\RCcomplexrho$.
Proceeding as in \ref{enu:r_epsIntForm}, we can also prove that $f'_{\eps}(z_{\eps}+h_{\eps})=f'_{\eps}(z_{\eps})+h_{\eps}\cdot f''_{\eps}(z_{\eps})+h_{\eps}\cdot r'_{\eps}(z_{\eps},h_{\eps})$.

\noindent\ref{enu:rLocLip}: It follows from \ref{enu:derivativefunction2}
and \ref{enu:GHF-GSF} that $f'$ is locally Lipschitz. The claim
then follows from \ref{enu:r_epsIntForm}.
\end{proof}
Using the natural definition of differentiability and Lem.~\ref{lem:littleoh},
it is possible to give intrinsic proof (i.e. without using nets of
functions that define given basic functions) of several classical
theorems of differential calculus, such as algebraic properties and
chain rule.
\begin{thm}
Let $U\subseteq\RCcomplexrho$ be a sharply open set, $f$, $g:U\ra\RCcomplexrho$
be basic functions, $z_{0}\in U$, and $c\in\RCcomplexrho$. If $f$,
$g$ are $\RCcomplexrho$-differentiable at $z_{0}$, then
\begin{enumerate}
\item $f+g$ is $\RCcomplexrho$-differentiable at $z_{0}$ and $(f+g)'(z_{0})=f'(z_{0})+g'(z_{0})$;
\item $cf$ is $\RCcomplexrho$-differentiable at $z_{0}$ and $(cf)'(z_{0})=cf'(z_{0})$;
\item $fg$ is $\RCcomplexrho$-differentiable at $z_{0}$ and $(fg)'(z_{0})=f'(z_{0})g(z_{0})+f(z_{0})g'(z_{0})$;
\item If $g(z_{0})$ is invertible, then $\left(\frac{f}{g}\right)'(z_{0})=\frac{f'(z_{0})g(z_{0})-g'(z_{0})f(z_{0})}{g(z_{0})^{2}}$.
\end{enumerate}
\end{thm}

\noindent For example, the following result proves, replicating the
classical proof, that GHF are closed with respect to composition and
hence to all holomorphic nonlinear operations:
\begin{thm}
\label{thm:chain}Let $U$, $V\subseteq\RCcomplexrho$ be sharply
open sets, $f:V\ra U$, $g:U\ra\RCcomplexrho$, and $z_{0}\in V$.
If $f$ is $\RCcomplexrho$-differentiable at $z_{0}$ and $g$ is
$\RCcomplexrho$-differentiable at $f(z_{0})$, then $(g\circ f)$
is $\RCcomplexrho$-differentiable at $z_{0}$ and 
\[
(g\circ f)'(z_{0})=g'(f(z_{0}))f'(z_{0}).
\]
\end{thm}

\begin{proof}
Since $f$ is $\RCcomplexrho$-differentiable at $z_{0}$ and $g$
is $\RCcomplexrho$-differentiable at $f(z_{0})$ then there exists
$U_{0}\subseteq\RCcomplexrho$ a sharply neigborhood of $0$, such
that $z_{0}+h\in V$ and $f(z_{0})+h\in U$ for all $h\in U_{0}$,
and 
\begin{align*}
f(z_{0}+h) & =f(z_{0})+f'(z_{0})h+o(h)\quad\text{and}\\
g(f(z_{0})+h) & =g(f(z_{0}))+g'(f(z_{0}))h+o(h)
\end{align*}
as $h\rightarrow0$. Therefore, there exist $\rho$-basic locally
Lipschitz functions $r_{1}(z_{0},-)$, $r_{2}(z_{0},-):U_{0}\ra\RCcomplexrho$
such that for all $h\in U_{0}$, we have
\begin{align*}
f(z_{0}+h) & =f(z_{0})+f'(z_{0})h+r_{1}(z_{0},h)h\\
g(f(z_{0})+h) & =g(f(z_{0}))+g'(f(z_{0}))h+r_{2}(z_{0},h)h.
\end{align*}
Then, for all $h\in U_{0}$ sufficiently small such that $f(z_{0}+h)\in U$,
we have
\begin{align*}
g(f(z_{0}+h)) & =g\left(f(z_{0})+f'(z_{0})h+r_{1}(z_{0},h)h\right)\\
 & =g(f(z_{0}))+g'(f(z_{0}))\left(f'(z_{0})h+r_{1}(z_{0},h)h\right)+r_{2}(z_{0},h)h\\
 & =g(f(z_{0}))+g'(f(z_{0}))f'(z_{0})h+\left[g'(f(z_{0}))r_{1}(z_{0},h)+r_{2}(z_{0},h)\right]h.
\end{align*}
The function $r(z_{0},h):=g'(f(z_{0}))r_{1}(z_{0},h)+r_{2}(z_{0},h)$,
then it satisfies \eqref{eq:generalizedComplexDiff1} of Def.~\ref{def:complexdiff},
from which the conclusion follows. 
\end{proof}
We close this section by characterizing sharply continuous functions
using the little-oh notations. We will use this result in the definition
of path integral in \cite{NuGi24a}. Using a little abuse of language,
in the following result we use the notion of strong little-oh $o(1)$,
which can be easily defined reformulating Def.~\ref{def:complexdiff}
with $m_{\eps}=0$ and replacing $h_{\eps}\cdot r_{\eps}(z_{\eps},h_{\eps})$
with $r_{\eps}(z_{\eps},h_{\eps})$, see also Def.~\ref{def:sheafTaylor}
below.
\begin{lem}
\label{lem:GC0LittleOh}Let $U\subseteq\RCcomplexrho$ be a sharply
open set, $f:U\ra\RCcomplexrho$ be a map and $z_{0}\in U$. Then
$f$ is sharply continuous at $z_{0}$ if and only if 
\[
f(z_{0}+h)-f(z_{0})=\bar{o}(1)\quad\text{as}\quad h\to0.
\]
The same relation holds for the strong little-oh $o(1)$ if $f$ is
a basic function defined by a net $f_{\eps}\in\mathcal{C}^{0}(U_{\eps})$
of continuous functions.
\end{lem}

\begin{proof}
We only have to define $r_{\eps}(z,h):=f_{\eps}(z+h)-f_{\eps}(z)$
for all $z\in U_{\eps}$ and all $h$ sufficiently small to show the
necessary condition for the strong little-oh, because the remaining
part follows directly from Def.~\ref{def:weakLittleoh} and the relations
between strong and weak little-oh.
\end{proof}

\subsection{\protect\label{subsec:epsHol}A criterion of generalized holomorphicity}

The next problem that we want to solve is an $\eps$-wise criterion
of generalized holomorphicity: If we start from an arbitrary net $f_{\eps}:U_{\eps}\ra\CC$
of holomorphic function (e.g.~a $\rho$-moderate net coming from
regularizing singularities in a given PDE, see Sec.~\ref{sec:Examples}),
when does this net define a GHF? 
\begin{thm}
\label{thm:epsDiff}Let $U_{\eps}\subseteq\CC$ be a net of open sets,
and $z_{\eps}\in U_{\eps}$ for all $\eps$. Let $(f_{\eps})$ be
a net of functions $f_{\eps}:U_{\eps}\ra\CC$ and $q\in\R_{>0}$.
Set $z:=[z_{\eps}]\in\RCcomplexrho$, and suppose that 
\begin{enumerate}
\item \label{enu:zInt}$B_{\diff\rho^{q}}(z)\subseteq\sint{U_{\eps}}$,
\item \label{enu:f_epsHolBas}$f_{\eps}\in\bigO(\Eball_{\rho_{\eps}^{q}}(z_{\eps}))$
and $(f_{\eps})$ defines a basic function of the type $B_{\diff\rho^{q}}(z)\ra\RCcomplexrho$,
\item \label{enu:f_eps'Mod}$\forall[c_{\eps}]\in B_{\diff\rho^{q}}(z):\ \left(f'_{\eps}(c_{\eps})\right)\in\CC_{\rho}$.
\end{enumerate}
Then $f:=\left[f_{\eps}(-)\right]|_{B_{\diff\rho^{q}}(z)}\in\ghf(B_{\diff\rho^{q}}(z))$.
In particular, this also shows that ordinary holomorphic functions
on the open set $\Omega\subseteq\CC$ are embedded as GHF with
\[
f\in\mathcal{O}(\Omega)\mapsto[f(-)]_{\rho}\in\ghf(d_{f}(\Omega)),
\]
where $z\in d_{f}(\Omega)$ if and only if for some $q\in\R_{>0}$,
conditions \ref{enu:zInt}, \ref{enu:f_epsHolBas} and \ref{enu:f_eps'Mod}
hold with $U_{\eps}=\Omega$. Note that if $f\in\bigO(\Omega)$, then
$d_{f}(\Omega)\subseteq\RCcomplexrho$ is sharply open, and it always
contains finite generalized points strictly contained in $\Omega$,
i.e.:
\[
\sint{\Omega}_{\text{\emph{fin}}}:=\left\{ z\in\sint{\Omega}\mid z\text{ is finite, }d(z,\partial\Omega)\in\R_{>0}\right\} \subseteq d_{f}(\Omega).
\]
\end{thm}

\begin{proof}
We proceed by proving that $f$ is $\RCcomplexrho$-differentiable
at $z$. To show that it is also $\RCcomplexrho$-differentiable at
any other $w\in B_{\diff\rho^{q}}(z)$ it suffices to consider a suitable
ball $B_{\diff\rho^{p}}(w)\subseteq B_{\diff\rho^{q}}(z)$ and argue
similarly. By assumption, $(f_{\eps})$ defines a basic function of
the type $B_{\diff\rho^{q}}(z)\ra\RCcomplexrho$, so that we have
to focus only on the remainder functions $r_{\eps}$ that we define
for all $\eps$ as
\begin{equation}
r_{\eps}(z,h):=\int_{0}^{1}f'_{\eps}(z+ht)\,\diff t-f'_{\eps}(z)\quad\forall z\in U_{\eps}:=\Eball_{\rho_{\eps}^{q}}(z_{\eps})\,\forall h\in U_{0\eps}:=\Eball_{\rho_{\eps}^{q}}(0).\label{eq:def-r_eps_holomophicity}
\end{equation}
Recall that $f_{\eps}\in\bigO(\Eball_{\rho_{\eps}^{q}}(z_{\eps}))$
by assumption. As we did above, considering the product $h_{\eps}r_{\eps}(z_{\eps},h_{\eps})$
and integrating, we get \eqref{eq:epswisecomplexdiff} and \eqref{eq:limitofresidue}
for all $[z_{\eps}]=z$ and all $[h_{\eps}]=h\in B_{\diff\rho^{q}}(0)$.
Therefore, it remains only to prove that the net $(r_{\eps})$ defines
a basic function of the type $\{z\}\times B_{\diff\rho^{q}}(0)\ra\RCcomplexrho$.
From \ref{enu:f_eps'Mod} and \eqref{eq:def-r_eps_holomophicity},
it follows that $(r_{\eps}(z_{\eps},h_{\eps}))\in\CC_{\rho}$. We
want to prove that $r(w,h):=[r_{\eps}(w_{\eps},h_{\eps})]$ is well-defined
for all $w=[w_{\eps}]\in B_{\diff\rho^{q}}(z)$ and all $h\in B_{\diff\rho^{q}}(0)$.
Take representatives $[\bar{w}_{\eps}]=[w_{\eps}]$, $[\bar{h}_{\eps}]=[h_{\eps}]$
and let $\gamma_{\eps}$ be any counterclockwise oriented circle centered
at $w_{\eps}$ and with radius $\left[s_{\eps}\right]\in\RCrealrho_{>0}$
contained in $\Eball_{\rho_{\eps}^{q}}(z_{\eps})$. By Cauchy formula
and the integral mean value theorem we have
\begin{align}
\left|f'_{\eps}(w_{\eps})-f'_{\eps}(\bar{w}_{\eps})\right| & \le\frac{s_{\eps}^{-2}}{2\pi}\int_{0}^{2\pi}\left|f_{\eps}(w_{\eps}+e^{it})-f_{\eps}(\bar{w}_{\eps}+e^{it})\right|\,\diff t\nonumber \\
 & =s_{\eps}^{-2}\left|f_{\eps}(w_{\eps}+e^{it_{\eps}})-f_{\eps}(\bar{w}_{\eps}+e^{it_{\eps}})\right|\nonumber \\
 & \le s_{\eps}^{-2}|f'_{\eps}(c_{\eps})|\cdot|w_{\eps}-\bar{w}_{\eps}|\label{eq:lip-f'}
\end{align}
for some $t_{\eps}\in[0,2\pi]$ and some $c_{\eps}\in[w_{\eps}+e^{it_{\eps}},\bar{w}_{\eps}+e^{it_{\eps}}]$.
This proves that $(f'_{\eps}(w_{\eps}))\simeq_{\rho}(f'_{\eps}(\bar{w}_{\eps}))$
(and the Lipschitz constant in \eqref{eq:lip-f'} depends only on
$f'_{\eps}$ and not on higher derivatives, so we can use assumption
\ref{enu:f_eps'Mod}). Similarly, we can prove independence on representatives
in the first integral summand in \eqref{eq:def-r_eps_holomophicity}.
\end{proof}
\noindent Comparing the previous Thm.~\ref{thm:epsDiff} with \cite[Thm.~2]{Oberguggenberger2007},
we can note several differences:
\begin{enumerate}[label=\arabic*)]
\item \cite[Thm.~2]{Oberguggenberger2007} considers only the case of holomorphicity
at ordinary points $z_{0}\in\Omega\subseteq\CC$. On the contrary,
in Thm.~\ref{thm:epsDiff}, in general, $\rho\ne(\eps)$ and $[z_{\eps}]_{\rho}$
can also be, e.g., an infinite generalized number. 
\item In assumption \ref{enu:f_eps'Mod}, we require that only the first
derivative is moderate and not estimates of the type $\left|f_{\eps}^{(n)}(z_{0})\right|\le n!\eta^{n+1}\eps^{-Q}$
($\eps$ small) with $\eta\in\R_{>0}$ for all derivatives. On the
one hand, this is due to our effort to start from a first order differentiability
condition \eqref{eq:epswisecomplexdiff}, that allows us to prove
\ref{enu:epsDerAreModerate} and \ref{enu:derivativefunction2} of
Thm.~\ref{thm:functionofderivative}. On the other hand, as we already
showed in the proof of Thm.~\ref{thm:functionofderivative}.\ref{enu:epsDerAreModerate},
we also have an estimate of the form $|f_{\eps}^{(n)}(z_{\eps})|\le n!a_{\eps}^{-n}\sup_{z\in\gamma}|f_{\eps}(z)|\le n!(\rho_{\eps}^{-Q})^{n}\rho_{\eps}^{-R}$,
but the term $\eta=\rho_{\eps}^{-Q}$ cannot usually be replaced by
a finite number.
\item This allows us to include a larger class of GHF with respect to the
classical Colombeau theory (in particular, the Dirac delta, see Sec.~\ref{sec:Examples}
below). One could also conceptually describe the present approach
as follows: depending on a given problem, let $\left(f_{\eps}\right)$
be a $\rho$-moderate net of holomorphic functions obtained by regularizing
the singularities of a given problem; find a gauge $\rho$ that satisfies
the assumption of Thm.~\ref{thm:epsDiff}, then $\overline{f}:=\left[f_{\eps}(-)\right]_{\rho}$
is a GHF that can be studied in the present theory.
\end{enumerate}
~

In general, it could happen that $z_{\eps}$ approaches $\partial U_{\eps}$
too quickly with respect to $\rho$, so that $(z_{\eps})$ does not
yield an interior point in the $\rho$-sharp topology. This problem
is solved by changing the gauge into a smaller one $\sigma_{\eps}\le\min\left\{ \rho_{\eps},d(z_{\eps},\partial U_{\eps})\right\} $,
which would hence allow us to measure smaller but still $\sigma$-invertible
infinitesimal numbers and thus greater infinite numbers.
\begin{defn}
\label{def:relGauges}Let $\sigma$, $\rho$ be two gauges, then we
say $\sigma\le\rho$ if $\forall^{0}\eps:\ \sigma_{\eps}\leq\rho_{\eps}$.
\end{defn}

\noindent The relation $(-)\leq(-)$ is reflexive, transitive, and
antisymmetric in the sense that $\sigma\leq\rho$ and $\rho\leq\sigma$
imply $\sigma_{\eps}=\rho_{\eps}$ for $\eps$ small. Clearly, $\sigma\leq\rho$
implies the inclusion of $\rho$-moderate nets $\R_{\rho}\subseteq\R_{\sigma}$.
We now need to link $\rho$-moderate numbers in $\RCcomplex\sigma$
with those in $\RCcomplex\rho$: If $\sigma\le\rho$ and $U\subseteq\RCcomplexrho$,
we define $\frontRiseDown{U}{\sigma}{\rho}U:=\left\{ \left[x_{\eps}\right]_{\sigma}\in\RCcomplex\sigma\mid(x_{\eps})\in\mathbb{C}_{\rho}\text{ and }\left[x_{\eps}\right]_{\rho}\in U\right\} $
as the set of $\RCcomplex\sigma$ numbers which are $\rho$-moderate
(recall that $[x_{\eps}]_{\sigma}$ is the equivalent class with respect
to the relation $\sim_{\sigma}$ as in Def.~\ref{def:RCring}). We
can also well-define a natural map $\iota:[x_{\eps}]_{\sigma}\in\RCcomplexud{\sigma}{\rho}\mapsto[x_{\eps}]_{\rho}\in\RCcomplex\rho$
because $\sigma\le\rho$. This map is surjective but generally not
injective, even if $\iota(x)=\iota(y)$ implies $|x-y|\le[\rho_{\eps}]_{\sigma}^{q}$
for all $q\in\R_{\ge0}$.

For instance, if $\rho=(\eps)$ and $\sigma=\left(e^{-\frac{1}{\eps}}\right)$,
then a generalized number of the form $\left[2^{-\frac{1}{\eps}}\right]_{\rho}=0$
but $\left[2^{-\frac{1}{\eps}}\right]_{\sigma}>0$. However, if $\left[x_{\eps}\right]_{\sigma}\in\RCcomplexud{\sigma}{\rho}$
and $\left[x_{\eps}\right]_{\rho}>0$ then $\left[x_{\eps}\right]_{\sigma}>0$.

Furthermore, we denote by $\rho$ and $\sigma$ two arbitrary gauges;
only when it will be needed, we will assume a relation between them,
such as $\sigma\leq\rho$.

In the following result, we show that a representative of a $\rho$-basic
locally Lipschitz function also defines a $\sigma$-basic locally
Lipschitz function,
\begin{lem}
\label{lem:sigmabasicf}Let $U\subseteq\RCcomplexrho$. Let $(U_{\eps})$
be a net of open subsets of $\CC$ such that $U\subseteq\left\langle U_{\eps}\right\rangle $
and $(f_{\eps})$ be a net of functions $f_{\eps}:U_{\eps}\ra\CC$.
If $(f_{\eps})$ defines a $\rho$-basic locally Lipschitz function
$f:U\ra\RCcomplexrho$, then for all $\sigma\leq\rho^{Q}$, for some
$Q\in\R_{>0}$, the net $(f_{\eps})$ defines a $\sigma$-basic locally
Lipschitz map $\overline{f}:=\left[f_{\eps}(-)\right]:\frontRiseDown{U}{\sigma}{\rho}U\ra\RCcomplex\sigma$.
\end{lem}

\begin{proof}
Let $[z_{\eps}]_{\sigma}=\overline{z}\in\frontRiseDown{U}{\sigma}{\rho}U$.
Then there exists $(x_{\eps})\in\CC_{\rho}$ such that $[x_{\eps}]_{\sigma}=\overline{z}=[z_{\eps}]_{\sigma}$
and $[x_{\eps}]_{\rho}\in U$. Therefore, we have $[f_{\eps}(x_{\eps})]_{\rho}\in\RCcomplexrho$
and $(f_{\eps}(x_{\eps}))\in\CC_{\rho}\subseteq\CC_{\sigma}$, hence
$[f_{\eps}(x_{\eps})]_{\sigma}\in\RCcomplex\sigma$. By \eqref{eq:lipcond},
there exist $L=\left[L_{\eps}\right]_{\rho}$, $r=\left[r_{\eps}\right]_{\rho}\in\RCrealrho_{>0}$
such that $|f(x)-f(y)|\leq L\cdot|x-y|$ whenever $x$, $y\in B_{r}(z)$.
By Lem.~\ref{lem:mayer}, there exist $N$, $\overline{q}\in\N$
such that $L_{\eps}\leq\rho_{\eps}^{-N}$ and $r_{\eps}\geq\rho_{\eps}^{\overline{q}}$,
for small $\eps$. Since $(x_{\eps})\sim_{\sigma}(z_{\eps})$ and
$\sigma\leq\rho$, then for sufficiently small $\eps$, we have 
\[
|f_{\eps}(x_{\eps})-f_{\eps}(z_{\eps})|\leq L_{\eps}|x_{\eps}-z_{\eps}|\leq\rho_{\eps}^{-N}\cdot\sigma_{\eps}^{\overline{q}}\leq\sigma_{\eps}^{q}\quad\forall q\geq\overline{q}-N.
\]
This implies $[f_{\eps}(z_{\eps})]_{\sigma}=\overline{f}(\overline{z})=[f_{\eps}(x_{\eps})]_{\sigma}$,
and hence conditions \ref{enu:basicF1} and \ref{enu:basicF2} follow.
To prove \eqref{eq:lipcond}, by setting $\overline{L}:=\left[L_{\eps}\right]_{\sigma}$
and $\overline{r}:=[r_{\eps}]_{\sigma}$, we have $\overline{L}$,
$\overline{r}\in\RCreal\sigma_{>0}$, and for every $x$, $y\in\frontRiseDown{U}{\sigma}{\rho}U\cap B_{\overline{r}}(\overline{z})$
\[
|\overline{f}(x)-\overline{f}(y)|\leq\overline{L}\cdot|x-y|.
\]
\end{proof}
A simple case of this construction is the exponential map
\[
e^{\left(-\right)}:[z_{\eps}]\in\left\{ z\in\RCcomplexrho\mid\exists c\in\RCrealrho_{>0}:|z|\leq\log c\right\} \mapsto\left[e^{z_{\eps}}\right]\in\RCcomplexrho.
\]
The domain of this map depends on the infinitesimal net $\rho$. For
instance, if $\rho=\left(\eps\right)$ then all its points are bounded
by generalized numbers of the form $\left[-N\log\eps\right]$, $N\in\N$;
whereas if $\rho=\left(e^{-\frac{1}{\eps}}\right)$, all points are
bounded by $\left[N\eps^{-1}\right]$, $N\in\N$. However, if we consider
a gauge $\sigma\leq\rho$ such that $\sigma_{\eps}:=\exp\left(-\rho_{\eps}^{\frac{1}{\eps}}\right)$,
then $\overline{e}^{\left(-\right)}:[z_{\eps}]\in\RCcomplexud{\sigma}{\rho}\mapsto\left[e^{z_{\eps}}\right]\in\RCcomplex\sigma$
is well defined.

\subsection{\protect\label{subsec:The-Cauchy-Riemann-equations}The Cauchy-Riemann
equations}

In this section, we want to prove that these GHF satisfy the Cauchy-Riemann
equations. We start from the definitions of $\RCrealrho$-differentiability
and $\RCrealrho$-partial differentiability. Even if in this series
of papers we will apply these definitions mainly to the real and imaginary
parts of a GHF (which are GSF, and hence all partial derivatives are
well-defined), these notions can be introduced independently from
the theory of GSF using the language of strong little-oh. Moreover,
notions of differentiability which are more general and independent
from GSF theory are used in the present work to extend Goursat, Looman-Menchoff
and Montel theorems to GHF.

We already pointed out, see Rem.~\ref{rem:complexdiff}.\ref{enu:sheafHol}
and Thm.~\ref{thm:functionofderivative}, the importance of using
the sheaf $\bigO(-)$ of holomorphic functions in Def.~\ref{def:complexdiff}.
Also in Lem.~\ref{lem:GC0LittleOh}, in the definition of strong
$o(1)$, it is clear that we are using the same idea, but with the
sheaf $\mathcal{C}^{0}(-)$ of continuous functions and, instead of
a first order Taylor difference $f_{\eps}(z_{\eps}+h_{\eps})-f_{\eps}(z_{\eps})-h_{\eps}\cdot m_{\eps}$,
a $0$-order Taylor difference. The following Def.~\ref{def:partDiff}
of $\RCrealrho$-(partially) differentiable maps follows the same
idea, but using the sheaf of (partially) differentiable maps $\R^{2}\ra\R^{2}$
and suitable Taylor (partial) differences.

In general, $T_{N}(u,z,h,\vec{v},(m_{j})_{j})$ denotes a (possibly
partial) Taylor difference (i.e.~the function $u$ evaluated at a
translated argument minus the corresponding Taylor polynomial) of
order $N\in\N$ of the function $u$ at the point $z$, with increment
$h$ in the direction $\vec{v}$ and coefficients $(m_{j})_{j=1}^{N}$,
and we have:
\begin{defn}
\label{def:sheafTaylor}Let $\mathbb{K}\in\{\R,\mathbb{C}\}$, $U\subseteq\RCKrho^{n}$,
$V\subseteq\RCKrho^{d}$, $T_{N}$ be a Taylor difference of order
$N$. Let $S(X)\subseteq(\R^{d})^{X}$ be a family of maps for all
open subsets $X\subseteq\field{K}^{n}$ (e.g.~a sheaf). Then, we
write $u\in\frontRise{\mathcal{G}}{\rho}\mathcal{G}(U,V;S(-),T_{N})$
if $u:U\ra V$ is a basic function and there exist $U_{0}\subseteq\RCKrho^{n}$
and nets $u_{\eps}\in S(U_{\eps})$, $r_{\eps}:U_{\eps}\times U_{0\eps}\ra\field{K}^{d}$
such that:
\begin{enumerate}
\item \label{enu:neighborhood-1}$U_{0}$ is a sharp neighborhood of $0$,
\item \label{enu:f_eps-1}$(u_{\eps})$ defines $u$,
\item \label{enu:r_eps-1}$(r_{\eps})$ defines a basic function $r:U\times U_{0}\ra\RCKrho^{d}$.
\end{enumerate}
For all $z\in U$, we have:
\begin{enumerate}[resume]
\item \label{enu:limit-r-1}$\hyplimf{\rho}{h}r(z,h)=0$ and there exists
$[m_{1\eps}]=m_{1},\ldots,[m_{N\eps}]=m_{N}\in\RCKrho^{d}$ such that
for all representatives $[z_{\eps}]=z\in U$ and all $[h_{\eps}]\in U_{0}$,
we have
\begin{align}
\forall^{0}\eps:\ T_{N}(u_{\eps},z_{\eps},h_{\eps},\vec{v}_{\eps},(m_{j\eps})_{j}) & =h_{\eps}^{N}\cdot r_{\eps}(z_{\eps},h_{\eps}).\label{eq:epswisecomplexdiff-1}\\
r_{\eps}(z_{\eps},0) & =0=\lim_{h\rightarrow0}{r_{\eps}(z_{\eps},h)}.\label{eq:limitofresidue-1}
\end{align}
\item \label{enu:highestDer}For all $\alpha\in\N^{n}$ with $|\alpha|=N$,
the directional derivative $[z_{\eps}]\in U\mapsto\left[\frac{\partial^{\alpha}u_{\eps}}{\partial\vec{v}_{\eps}^{N}}(z_{\eps})\right]\in\RCKrho^{d}$
is well-defined and sharply continuous.
\end{enumerate}
If conditions \ref{enu:neighborhood-1} - \ref{enu:limit-r-1} hold,
we simply write
\[
T_{N}(u,z,h,\vec{v},(m_{j})_{j})=o(h^{N})\text{ as }h\to0\text{ in }z\in U,
\]

\noindent and we say that $u$ is $T_{N}$\emph{-regular with representatives
}$u_{\eps}\in S(U_{\eps})$ (\textit{and with respect to} \textit{$\RCKrho^{n}$}
in case we have to underscore the dependence from $\RCKrho^{n}$).
Usually, specifying the family $S(X)$ it is already implicitly clear
what is the Taylor difference $T_{N}$ we are considering. As already
proved in Thm.~\ref{thm:functionofderivative}, in case of GHF, the
last condition \ref{enu:highestDer} can be proved from the remaining
ones.
\end{defn}

\noindent We have the following examples:
\begin{defn}
\label{def:partDiff}Let $V\subseteq U\subseteq\RCrealrho^{2}$ be
a sharply open set, $f:U\ra\RCrealrho^{2}$ be a basic function. Then,
\begin{enumerate}
\item A function $u$ is said to be $\RCrealrho$-\emph{differentiable on}
$V$ if for all $(x,y)\in V$ there exist a linear $\rho$-basic function
$D_{(x,y)}:\RCrealrho^{2}\ra\RCrealrho^{2}$ satisfying the condition
$u(x+h_{1},y+h_{2})-u(x,y)-D_{(x,y)}(h)=o(h)$ as $h=(h_{1},h_{2})\rightarrow0$
and with representatives in:
\[
\text{Diff}(U):=\left\{ u:\bar{U}\ra\R^{d}\mid u\text{ is differentiable on }\bar{U},\;U\subseteq\left\langle \bar{U}\right\rangle \right\} .
\]
This unique $D_{(x,y)}$ (see Rem.~\ref{rem:diff}.\ref{enu:uniqueDiff}
below) is called the differential of $f$ at $(x,y)$ and is denoted
by $f'(x,y)$ or $\diff f(x,y)$.
\item A function $f$ is said to be $\RCrealrho$-\emph{partially differentiable}
with respect to $x$ (resp.~$y$) on $V$ if for all $(x,y)\in V$
there exist $m_{x}$ (resp.~$m_{y}$)$\in\RCrealrho$ such that $f(x+h,y)=f(x,y)+h\cdot m_{x}+o(h)$
as $h\rightarrow0$ (resp. $f(x+h,y)=f(x+y)+h\cdot m_{x}+o(h)$ as
$h\rightarrow0$) and with representatives in:
\[
\text{PDiff}(U):=\left\{ u:\bar{U}\ra\R^{d}\mid\forall j=1,\ldots,n:\ \exists\partial_{j}u\text{ on }\bar{U},\;U\subseteq\left\langle \bar{U}\right\rangle \right\} .
\]
We denote the $\RCrealrho$-partial derivative of $f$ with respect
to $x$ (resp.~$y$) as $\frac{\partial f}{\partial x}(x,y):=\partial_{1}f(x,y):=\partial_{x}f(x,y):=m_{x}$
(or $\frac{\partial f}{\partial y}(x,y):=\partial_{2}f(x,y):=\partial_{2}f(x,y):=m_{y}$).
\end{enumerate}
\end{defn}

\begin{rem}
~\label{rem:diff}
\begin{enumerate}[label=\alph*)]
\item Linear $\rho$-basic functions $D:\RCrealrho^{2}\ra\RCrealrho^{2}$
are precisely the $\rho$-basic functions that can be written as $D(z)=Az$,
for all $z\in\RCrealrho^{2}$, where $A\in M_{2\times2}(\RCrealrho)$
is a matrix with entries in the ring $\RCrealrho$. That is, each
linear $\rho$-basic function $D$ is a multiplication by a unique
$A\in M_{2\times2}(\RCrealrho)$ and conversely.
\item \label{enu:uniqueDiff}The only weak $\bar{o}(h)$-linear function
$D$ on $\RCrealrho^{2}$, i.e.~such that it can be written as $D(h)=h\cdot r(h)$
with $r(h)\to0$, is the zero function. In fact, assume that $D(h_{0})\ne0$
for some $h_{0}\in\RCrealrho^{2}$. From \eqref{eq:orderedFieldSub},
there exists $L\subzero I$ and $r\in\N$ such that $|D(h_{0})|>_{L}0$
is invertible. We cannot have $h_{0}=_{L}0$ because, otherwise, $D(h_{0})=_{L}0$
by linearity. Once again from \eqref{eq:orderedFieldSub} (applied
to the ring $\RCrealrho|_{L})$ there exists $K\subzero L$ such that
both $|h_{0}|>_{K}0$ and $|D(h_{0})|>_{K}0$ are invertible. For
all invertible $\alpha\in\RCrealrho|_{K}$, we have
\[
D(\alpha h_{0})=_{K}\alpha D(h_{0})=_{K}\alpha h_{0}r(\alpha h_{0}),
\]
and we get the contradiction $|D(h_{0})|/|h_{0}|=_{K}|r(\alpha h_{0})|\to0$
as $\alpha\to0$ in $\RCrealrho|_{K}$.
\end{enumerate}
\end{rem}

In the following, we prove the equivalent condition for differentiability
on partial derivatives. Thanks to our language of little-ohs, the
proof is formally identical to the classical one, and we hence omit
it.
\begin{thm}
\label{thm:eqconditionofdifferentiability}Let $V\subseteq\RCrealrho^{2}$
be a sharply open set and $z=(z_{1},z_{2})\in V$
\begin{enumerate}
\item If $f:=(f_{1},f_{2}):V\ra\RCrealrho^{2}$ be a $\RCrealrho$-differentiable
at $z$, then for $i$,$j\in\left\{ 1,2\right\} $, there exists the
$\RCrealrho$-partial derivative $\partial_{j}f_{i}$ and is sharply
continuous. Furthermore,
\[
f'(z)=\left[\begin{array}{cc}
\partial_{1}f_{1}(z) & \partial_{2}f_{1}(z)\\
\partial_{1}f_{2}(z) & \partial_{2}f_{2}(z)
\end{array}\right].
\]
\item If $f:=(f_{1},f_{2}):V\ra\RCrealrho^{2}$ be a basic function. Suppose
that for each $i$, $j\in\{1,2\}$, the $\RCrealrho$-partial derivative
$\partial_{j}f_{i}$ exists on some neighbourhood of $z$ and the
$\RCrealrho$-partial derivative $\partial_{j}f_{i}$ is sharply continuous
at $z$. Then $f$ is $\RCrealrho$-differentiable at $z$.
\end{enumerate}
\end{thm}

We can finally prove the CRE, and the related Goursat-like theorems.
The proof of the CRE is again formally identical to the classical
one. For Goursat theorem, we need to take representatives in
\[
\text{CRE}(U):=\left\{ u\in\text{PDiff}(U)\mid\text{\ensuremath{\partial_{j}u_{k}}\text{ satisfy the CRE}, }j,k=1,2\right\} .
\]

\begin{thm}[Cauchy-Riemann equations and Goursat theorem]
\label{thm:CRE}Consider a sharply open set $U\subseteq\RCrealrho^{2}$,
set $\overline{U}:=\left\{ z\in\RCcomplexrho:\left(\realpart(z),\impart(z)\right)\in U\right\} $
and take $(x_{0},y_{0})\in U$. Let $u$, $v:U\ra\RCrealrho^{2}$
be maps, and for all $z=x+iy\in\overline{U}$ set $f(z):=u(x,y)+iv(x,y)$.
Then, we have:
\begin{enumerate}
\item \label{enu:CRE}If $f$ is $\RCcomplexrho$-differentiable at $z_{0}:=x_{0}+iy_{0}$,
then $u$, $v$ are sharply continuously $\RCrealrho$-differentiable
at $(x_{0},y_{0})$ and satisfy the CRE:
\begin{equation}
\partial_{1}u=\partial_{2}v\quad\text{and}\quad\partial_{2}u=-\partial_{1}v.\label{eq:cauchyriemann}
\end{equation}
\item \label{enu:Goursat}Let $u$, $v$ be $\RCrealrho$-partially differentiable
functions on $U$ with representatives in
\begin{equation}
\left\{ u\in\text{\emph{CRE}}(U_{\eps})\cap\mathcal{C}^{0}(U_{\eps})\mid\text{\ensuremath{\partial_{j}u_{k}\in\mathcal{C}^{0}(U_{\eps})}, }j,k=1,2\right\} .\label{eq:reprGoursat}
\end{equation}
Assume that the generalized partial derivatives of $u$, $v$ satisfy
the CRE \eqref{eq:cauchyriemann}, then $f\in\ghf(U)$.
\item \label{enu:strongGoursat}Let $u$, $v$ be $\RCrealrho$-partially
differentiable functions on $U\subseteq\RCrealrho^{2}$ with representatives
in $\text{\emph{CRE}}(U_{\eps})\cap\text{\emph{Diff}}(U_{\eps})$.
Assume that the generalized partial derivatives of $u$, $v$ satisfy
the CRE \eqref{eq:cauchyriemann}, then $f\in\ghf(U)$.
\end{enumerate}
\end{thm}

\begin{proof}
\ref{enu:CRE}: Assume that $f$ is $\RCcomplexrho$-differentiable
at $z_{0}$. By Thm.~\ref{thm:eqconditionofdifferentiability}, the
partial derivatives $\partial_{j}u$ and $\partial_{j}v$ ($j=1$,
$2$) exist and are sharply continuous. Since $f$ is $\RCcomplexrho$-differentiable
at $z_{0}$, then there exists $f'(z_{0})\in\RCcomplexrho$ satisfying
\[
f(z_{0}+h)=f(z_{0})+m\cdot h+o(h)\quad\text{as}\quad h\rightarrow0.
\]
Taking $h=h_{1}\in\RCrealrho$, we have
\[
\begin{aligned}f(z_{0}+h_{1})-f(z_{0}) & =u(x_{0}+h_{1},y_{0})+iv(x_{0}+h_{1},y_{0})-u(x_{0},y_{0})-iv(x_{0},y_{0})\\
 & =u(x_{0}+h_{1},y_{0})-u(x_{0},y_{0})+i\left[v(x_{0}+h_{1},y_{0})-v(x_{0},y_{0})\right]\\
 & =\partial_{1}u(x_{0},y_{0})h_{1}+o(h_{1})+i\left[v(x_{0},y_{0})h_{1}+o(h_{1})\right]\\
 & =\left[\partial_{1}u(x_{0},y_{0})+i\partial_{1}v(x_{0},y_{0})\right]h_{1}+o(h_{1})\quad\text{as}\quad h_{1}\rightarrow0.
\end{aligned}
\]
Since the derivative is unique, then 
\begin{equation}
f'(z_{0})=\partial_{1}u(x_{0},y_{0})+i\partial_{1}v(x_{0},y_{0}).\label{eq:CRE1}
\end{equation}
Taking $h=ih_{2}$, $h_{2}\in\RCrealrho$ we have
\[
\begin{aligned}f(z_{0}+ih_{2})-f(z_{0}) & =u(x_{0},y_{0}+h_{2})+iv(x_{0},y_{0}+h_{2})-u(x_{0},y_{0})-iv(x_{0},y_{0})\\
 & =u(x_{0},y_{0}+h_{2})-u(x_{0},y_{0})+i\left[v(x_{0},y_{0}+h_{2})-v(x_{0},y_{0})\right]\\
 & =\partial_{2}u(x_{0},y_{0})h_{2}+o(h_{2})+i\left[v(x_{0},y_{0})h_{2}+o(h_{2})\right]\\
 & =\left[\partial_{2}u(x_{0},y_{0})+i\partial_{2}v(x_{0},y_{0})\right]h_{2}+o(h_{2})\quad\text{as}\quad h_{2}\rightarrow0.
\end{aligned}
\]
On the other hand, we have
\[
\begin{aligned}f(z_{0}+ih_{2})-f(z_{0}) & =f'(z_{0})ih_{2}+o(h_{2})\quad\text{as}\quad h_{2}\rightarrow0\end{aligned}
.
\]
Since the derivative is unique, then 
\begin{equation}
f'(z_{0})=\frac{1}{i}\left(\partial_{2}u(x_{0},y_{0})+i\partial_{2}v(x_{0},y_{0})\right).\label{eq:CRE2}
\end{equation}
By \eqref{eq:CRE1} and \eqref{eq:CRE2}, we hence have 
\[
\partial_{1}u(x_{0},y_{0})+i\partial_{1}v(x_{0},y_{0})=\partial_{2}v(x_{0},y_{0})-i\partial_{2}u(x_{0},y_{0})
\]
and from which the conclusion follows.

\noindent\ref{enu:Goursat}: In order to prove that $f\in\ghf(U)$,
we have to show that $f_{\eps}\in\bigO(U_{\eps})$ (and here we clearly
have to use the classical Goursat theorem, see e.g.~\cite[Thm.~1]{GrMo}),
and the strong little-oh differentiability condition \eqref{eq:generalizedComplexDiff1}.

To prove that $f_{\eps}\in\bigO(U_{\eps})$, take representatives
$u=[u_{\eps}(-)]$ and $v=[v_{\eps}(-)]$ in \eqref{eq:reprGoursat}.
Then, $\partial_{j}u_{\epsilon}$ and $\partial_{j}v_{\eps}$ exist
and are continuous in $U_{\eps}$, they satisfy there the CRE, and
$f_{\eps}:=u_{\eps}+iv_{\eps}$ is continuous because $u_{\eps}$,
$v_{\eps}\in\mathcal{C}^{0}(U_{\eps})$. Therefore, for the classical
Goursat theorem \cite[Thm.~1]{GrMo}, we have the conclusion.

To show \eqref{eq:generalizedComplexDiff1}, by Thm.~\ref{thm:eqconditionofdifferentiability},
we have that the $\RCrealrho$-partial derivatives $\partial_{i}u$
and $\partial_{i}v$ exist at $(x_{0},y_{0})$, $i=1,2$. Hence, for
all $h=h_{1}+ih_{2}\in\RCcomplexrho$ sufficiently small, by applying
Lemma \ref{lem:littleoh} and the CRE, we obtain

\begin{eqnarray}
f(z_{0}+h) & = & u(x_{0}+h_{1},y_{0}+h_{2})+iv(x_{0}+h_{1},y_{0}+h_{2})\nonumber \\
 & = & \left[u(x_{0},y_{0})+\partial_{1}u(x_{0},y_{0})h_{1}+\partial_{2}u(x_{0},y_{0})h_{2}+o(h)\right]\nonumber \\
 &  & +i\left[v(x_{0},y_{0})+\partial_{1}v(x_{0},y_{0})h_{1}+\partial_{2}v(x_{0},y_{0})h_{2}+o(h)\right]\nonumber \\
 & = & \left[u(x_{0},y_{0})+iv(x_{0},y_{0})\right]+\left[\partial_{1}u(x_{0},y_{0})+i\partial_{1}v(x_{0},y_{0})\right]h_{1}\nonumber \\
 &  & +\left[\partial_{2}u(x_{0},y_{0})+i\partial_{2}v(x_{0},y_{0})\right]h_{2}+o(h)\nonumber \\
 & = & \left[u(x_{0},y_{0})+iv(x_{0},y_{0})\right]+\left[\partial_{1}u(x_{0},y_{0})-i\partial_{2}u(x_{0},y_{0})\right]h_{1}\nonumber \\
 &  & +\left[\partial_{1}u(x_{0},y_{0})-i\partial_{2}u(x_{0},y_{0})\right]ih_{2}+o(h)\nonumber \\
 & = & f(z_{0})+\left[\partial_{1}u(x_{0},y_{0})-i\partial_{2}u(x_{0},y_{0})\right]h+o(h)\text{ as }h\rightarrow0.\label{eq:CRE-diff}
\end{eqnarray}

\noindent\ref{enu:strongGoursat}: We can proceed as above using
\cite[Thm.~2]{GrMo}.
\end{proof}
Proceeding like in these extensions of Goursat theorem and using the
classical Looman-Menchof theorem \cite[Thm.~3]{GrMo}, we obtain the
following
\begin{thm}[Looman-Menchof]
In the general assumptions of Thm.~\ref{thm:CRE}, let $u$, $v$
be $\RCrealrho$-partially differentiable functions on $U$ with representatives
in $\text{\emph{CRE}}(U_{\eps})\cap\mathcal{C}^{0}(U_{\eps})$, then
$f\in\ghf(U)$.
\end{thm}

\noindent If we only assume that $u$, $v$ are $\RCrealrho$-partially
differentiable functions on $U$ with representatives in $\text{CRE}(U_{\eps})$
and $f$ is sharply bounded on a closed ball, we obtain the following
generalized version of Montel theorem:
\begin{thm}[Montel]
In the general assumptions of the previous Thm.~\ref{thm:CRE},
let $u$, $v$ be $\RCrealrho$-partially differentiable functions
on $U$ with representatives in $\text{\emph{CRE}}(U_{\eps})$, and
assume that $f$ is sharply bounded on the closed ball $\overline{B}_{r}(z_{0})\subseteq U$,
i.e.
\begin{equation}
\exists M\in\RCrealrho_{>0}\,\forall z\in\overline{B}_{r}(z_{0}):\ |f(z)|<M.\label{eq:sharpBounded}
\end{equation}
Then $f\in\ghf(B_{r}(z_{0}))$.
\end{thm}

\begin{proof}
In fact, take representatives $u_{\eps}$, $v_{\eps}\in\text{CRE}(U_{\eps})$
of $u$, $v$ resp., and set $f_{\eps}:=u_{\eps}+iv_{\eps}$. Let
$M=[M_{\eps}]$ from \eqref{eq:sharpBounded}. Proceeding by contradiction
as in \eqref{eq:forall-h}, we can prove that
\[
\forall^{0}\eps\,\forall z\in\Eball_{r_{\eps}}(z_{0\eps}):\ |f_{\eps}(z)|<M_{\eps}.
\]
For these small $\eps$, we can hence apply the classical Montel theorem
to prove that $f_{\eps}\in\bigO(\Eball_{r_{\eps}}(z_{0\eps}))$. We
can finally argue as in \eqref{eq:CRE-diff}.
\end{proof}

\section{\protect\label{sec:Examples}Embedding of compactly supported distributions
and other examples}

The question of embedding suitable subspaces of $\mathcal{C}^{0}(\Omega)$,
$\mathcal{C}^{\infty}(\Omega)$, $\mathcal{D}'(\Omega)$ (where $\Omega\subseteq\R^{2}$
is an open subset) in spaces of GHF arises naturally. The embedding
of classical holomorphic functions $\mathcal{O}(\Omega)$ into GHF
has already been considered in Thm.~\ref{thm:epsDiff} by the canonical
map
\begin{equation}
\sigma:\mathcal{O}(\Omega)\rightarrow\ghf(\sint{\Omega}_{\text{fin}})\quad f\mapsto\left[f(-)\right]_{\rho}\label{eq:canonical}
\end{equation}
which is an injective homomorphism of algebras.

As we already mentioned above, even the embedding of Dirac delta $\delta$
and Heaviside function $H$ would allow us to consider powers $\delta^{k}$,
$H^{h}$, $k$, $h\in\N$, and compositions such as $\delta\circ\delta$,
$\delta^{k}\circ H^{h}$, $H^{h}\circ\delta^{k}$.

Classically, the idea is to associate a net of holomorphic functions
$(f_{\eps})$ to a given distribution by a regularization process
using an \emph{entire} function $\mu\in\Coo(\R^{2})$ such that
\begin{equation}
\int_{\R^{2}}\mu(x)\,\diff x=1.\label{eq:int1}
\end{equation}
For example, we can take $\mu:=\mathcal{F}^{-1}(\beta)\in\mathcal{S}(\R^{2})$
as the inverse Fourier transform of any smooth compactly supported
function $\beta\in\mathcal{D}(\R^{2})$ with $\beta(0)=1$. The Schwartz-Paley-Wiener
theorem implies that $\mu$ is an entire function that can be analytically
extended to the whole complex space. We then define a $\delta$-net
as an approximation to the identity by 
\begin{equation}
\mu_{\eps}(z):=\rho_{\eps}^{-2}\cdot\mu(\rho_{\eps}^{-1}\cdot z)\quad\forall z\in\CC\,\forall\eps.\label{eq:approxId}
\end{equation}
For every compactly supported complex valued distribution $T=(T_{1},T_{2})\in\mathcal{E}'(\R^{2},\CC)$,
the convolution $T*\mu_{\eps}$ is well-defined and given by 
\[
(T*\mu_{\eps})(z)=\langle T,\mu_{\eps}(z-\cdot)\rangle\quad\forall z\in\CC.
\]
This convolution defines an entire holomorphic function since $T$
is compactly supported and $\mu_{\eps}$ is an entire holomorphic
function, see e.g.~\cite[Exercise 27.2~(4)]{Tre1967}. Using a little
abuse of language, we call \emph{entire mollifier }any function $\mu\in\Coo(\R^{2})$
such that \eqref{eq:int1} (without any requirement about its support,
positivity or symmetry).
\begin{thm}
\label{thm:embeddingdistribution}Let $\mu\in\Coo(\R^{2})$ be a entire
mollifier, define the $\rho$-approximation to the identity $(\mu_{\eps})$
as above in \eqref{eq:approxId}. Then
\[
\iota:T\in\mathcal{E}'(\R^{2},\CC)\mapsto\left[(T*\mu_{\eps})(-)\right]\in\ghf(\RCcomplexrho)
\]
is a linear monomorphism. Moreover, the embedding $\iota$ preserves
derivatives:
\begin{equation}
\iota(\partial^{\alpha}T)=\partial^{\alpha}\iota(T)\quad\forall\alpha\in\N^{2},\label{eq:der}
\end{equation}
and hence for all $z=x_{0}+iy_{0}\in\RCcomplexrho$
\begin{align}
\iota(T)'(z) & =\iota(\partial_{1}T_{1})(x_{0},y_{0})+i\iota(\partial_{1}T_{2})(x_{0},y_{0})\nonumber \\
 & =\iota(\partial_{2}T_{2})(x_{0},y_{0})-i\iota(\partial_{2}T_{1})(x_{0},y_{0}).\label{eq:CRE-distr}
\end{align}
\end{thm}

\begin{proof}
By Thm.~\ref{thm:epsDiff}, we have to show that for all $z=[z_{\eps}]\in\RCcomplexrho$,
$T*\mu_{\eps}$ defines a basic function of the type $\RCcomplexrho\ra\RCcomplexrho$,
and $\left((T*\mu_{\eps})'(c_{\eps})\right)\in\CC_{\rho}$ for all
$[c_{\eps}]\in\RCcomplexrho$. Since $T$ is compactly supported,
we can write $T=\sum_{|\alpha|\le m}\partial^{\alpha}g_{\alpha}$
for $m\in\N$ and compactly supported functions $g_{\alpha}\in\mathcal{C}^{0}(\R^{2})$.
We therefore have
\begin{align}
(T*\mu_{\eps})(z_{\eps}) & =\sum_{|\alpha|\le m}(-1)^{|\alpha|}\int_{\R^{2}}g_{\alpha}(x)\partial^{\alpha}\mu_{\eps}(z_{\eps}-x)\,\diff x\label{eq:wellDef}\\
 & =\sum_{|\alpha|\le m}(-1)^{|\alpha|}\int_{\R^{2}}g_{\alpha}(z_{\eps}-y)\partial^{\alpha}\mu_{\eps}(y)\,\diff x\nonumber \\
 & =\sum_{|\alpha|\le m}(-1)^{|\alpha|}\int_{\R^{2}}g_{\alpha}(z_{\eps}-\rho_{\eps}w)\rho_{\eps}^{-|\alpha|}\partial^{\alpha}\mu(w)\,\diff w\nonumber \\
 & =\mathcal{O}(\rho_{\eps}^{-m}).\nonumber 
\end{align}
Since $(T*\mu_{\eps})'=T*\mu_{\eps}'$, the same argument applies
to the derivatives. Finally, from \eqref{eq:wellDef} it also follows
independence from representatives since the same property holds for
$[\mu_{\eps}(\cdot-x)]$. To show injectivity, assume that $\left[(T*\mu_{\eps})(-)\right]=0$.
Since $T*\mu_{\eps}\to T$ in $\mathcal{D}'(\R^{2})$, in order to
prove that $T=0$, it suffices to demonstrate that $T*\mu_{\eps}\to0$
as $\eps\to0$ uniformly on compact subsets. Suppose this were not
the case, so that we could find some compact $L\Subset\R^{2}$, some
$c>0$, a sequence $\eps_{k}\downarrow0$ and $z_{k}\in L$ such that
$|\left(T*\mu_{\eps}\right)(z_{k})|\ge c$ for all $k$. Fixing any
$z\in\R^{2}$ and setting $z_{\eps}:=z_{k}$ for $\eps=\eps_{k}$
and $z_{\eps}=z$ otherwise, we define an element $[z_{\eps}]\in\RCcomplexrho$
with $\left[(T*\mu_{\eps})(z_{\eps})\right]\ne0$, a contradiction.
Note that the generalized point $[z_{\eps}]$ is finite, and hence
the embedding of $T$ is actually uniquely determined by its values
on finite points of $\RCcomplexrho$.

For \eqref{eq:der}, we have
\[
\iota(\partial^{\alpha}T)_{\eps}=\left(\partial^{\alpha}T\right)*\mu_{\eps}=\partial^{\alpha}\left(T*\mu_{\eps}\right)=\partial^{\alpha}\iota(T)_{\eps}.
\]
Therefore, \eqref{eq:CRE-distr} follows from \eqref{eq:der} and
the CRE for GHF, i.e.~Thm.~\ref{thm:CRE}.
\end{proof}
In particular, if we assume that $T=\delta$, then we have a Dirac
delta embedded as a GHF, defined by
\[
\delta_{\eps}(z):=\left(\delta*\mu_{\eps}\right)(z)=\mu_{\eps}(z)=\rho_{\eps}^{-2}\mu\left(\rho_{\eps}^{-1}z\right).
\]
Since $\mu$ is an entire function, this also clearly confirms that
Dirac delta is in fact a GHF given by $\delta(z)=\diff\rho^{-2}\mu\left(\frac{z}{\diff\rho}\right)\in\RCcomplexrho$
and defined on $\RCcomplexrho$, and it allows us to include in this
theory a large family of interesting examples, sometime informally
studied in physics (cf.~\cite{Brewster2018,Bag,Dav,FrPS,Kai,Lin,Sma}).

From \eqref{eq:int1} and the identity theorem for $\mu$, it follows
that $\int_{\R}\mu=:c\ne0$. Therefore, we can also consider $\mu_{1}(x):=\frac{1}{c}\mu(x)$
for all $x\in\R$. Then $\delta_{1}(x):=\diff\rho^{-1}\mu_{1}\left(\frac{x}{\diff\rho}\right)$
for all $x\in\RCrealrho$ is defined by the $1$-dimensional $\delta$-net
$\rho_{\eps}^{1}\mu_{1}(\rho_{\eps}^{-1}x)\to\delta$ in $\mathcal{D}'(\R)$
(and it can also be considered the $\mu_{1}$-embedding of the Dirac
delta as a GSF; this formula firstly appears in \cite{Aragona2005}).
This also yields the link to the GHF $\delta(z)$: in fact $\delta(x)=c\diff\rho^{-1}\delta_{1}(x)$
for all $x\in\RCrealrho$.

Note also that $\delta(0)=\frac{\mu(0)}{\diff\rho^{2}}$ is an infinite
number. Moreover, if we take $\mu\in\mathcal{S}(\R^{2})$ and $x\in\RCrealrho$
``far'' from the origin, i.e.~$|x|\ge r\in\R_{>0}$, then $\delta(x)=0$
in $\RCcomplexrho$ because $|\delta(x)|\le C\cdot r^{-n}\cdot\diff\rho^{n-2}$
for some $C\in\R_{>0}$ and for all $n\in\N$.

Integral properties of this Dirac delta will be considered in \cite{NuGi24a};
sufficient conditions for the identity theorem will be explored in
\cite{NuGi24b} (see Sec.~\ref{subsec:Comparison-with-Colombeau}
below for some initial considerations).\\

The entire mollifier $\mu$ is another free variable of our mathematical
structure and it has to be chosen depending on the regularization
properties we need, see also Sec.~\ref{subsec:Comparison-with-Colombeau}
for a more complete account on this idea.\\

A second way, inspired by \cite{GrMo}, to get a meaningful class
of examples of GHF is to regularize a locally integrable function
convolving with a \emph{compactly supported} mollifier. However, since
necessarily we are not smoothing with an entire mollifier, we have
to assume that distributional CRE hold:
\begin{thm}
\label{thm:emb2}Let $\Omega$ be an open subset of $\CC$. Let $k\in\mathcal{D}(\Omega)$
with $\int_{\R^{2}}k=1$, and define $k_{\eps}(z):=\rho_{\eps}^{-2}\cdot k(\rho_{\eps}^{-1}\cdot z)$
for all $z\in\CC$. If $f$ is locally integrable on $\Omega$ and,
as a distribution, satisfies the CRE, then $f$ can be identified
with the GHF defined by
\[
\jmath(f):[z_{\eps}]\in d_{f}\mapsto\left[(f*k_{\eps})(z_{\eps})\right]\in\RCcomplexrho,
\]
in the sense that $\jmath(f)=\jmath(g)$ implies $f=g$. We recall
that the domain $d_{f}$ has been defined in Thm.~\ref{thm:epsDiff}
(in the present case with $U_{\eps}=\Omega$). This embedding preserved
weak derivatives of $f$, assuming that they exists: $\jmath(\partial^{\alpha}f)=\partial^{\alpha}\jmath(T)$.

\noindent Moreover, if $f\in\mathcal{C}^{0}(\Omega)$, then
\begin{enumerate}
\item \label{enu:stPartFun}The standard part of this embedding yields back
the given function, which necessarily is holomorphic: $\st{\jmath(f)}|_{\Omega}=f\in\bigO(\Omega)$,
\item \label{enu:noNonArch}If the standard part does not attains non-Archimedean
values, i.e.~$\jmath(f)|_{\Omega}\subseteq\CC$, then $\jmath(f)|_{\Omega}=f\in\bigO(\Omega)$.
\end{enumerate}
\end{thm}

\begin{proof}
We have to use again Thm.~\ref{thm:epsDiff}, but in order to show
that $f_{\eps}:=f*k_{\eps}\in\bigO(\Eball_{\rho_{\eps}^{q}}(z_{\eps}))$
we have to use the ideas of \cite{GrMo} (and \cite{Zal}) and the
assumption that CRE holds for weak derivatives of $f$. For the sake
of completeness, we quickly repeat here the arguments of \cite{GrMo}.

The function $f_{\eps}$ is smooth because $k_{\eps}$ is smooth.
Let $\frac{\partial}{\partial\overline{z}}=\frac{1}{2}\left(\frac{\partial}{\partial x}+i\frac{\partial}{\partial y}\right)$,
so that asserting that $f$ as a distribution satisfies CRE means
that
\begin{align*}
\frac{\partial f_{\eps}}{\partial\overline{z}} & =\frac{\partial}{\partial\overline{z}}\int_{\R^{2}}f(\zeta)\mu_{\eps}(z-\zeta)\,\diff\zeta=\int_{\R^{2}}f(\zeta)\frac{\partial}{\partial\overline{z}}\mu_{\eps}(z-\zeta)\,\diff\zeta\\
 & =-\int_{\R^{2}}f(\zeta)\frac{\partial}{\partial\overline{\zeta}}\mu_{\eps}(z-\zeta)\,\diff\zeta=0,
\end{align*}
because $\mu_{\eps}\in\mathcal{D}(\Omega)$ is a test function. Thus,
$f_{\eps}$ is a $\mathcal{C}^{\infty}$-function satisfying the CRE
and hence is a holomorphic function. We can now proceed exactly as
in Thm.~\ref{thm:embeddingdistribution} in order to apply Thm.~\ref{thm:epsDiff}.
Also the preservation of weak derivatives follows like in Thm.~\ref{thm:embeddingdistribution}
.

If $f\in\mathcal{C}^{0}(\Omega)$, then $f_{\eps}\to f$ in $\Omega$
as $\eps\to0^{+}$. Consider a standard point $z\in\Omega$ and taking
as $\gamma$ the boundary of any $\Eball_{r}(z)\subseteq\Omega$,
by Cauchy's integral formula, we have
\[
f(z)=\lim_{\eps\rightarrow0^{+}}f_{\eps}(z)=\lim_{\eps\rightarrow0^{+}}\frac{1}{2\pi i}\int_{\gamma}\frac{f_{\eps}(\zeta)}{\zeta-z}\,\diff\zeta=\frac{1}{2\pi i}\int_{\gamma}\frac{f(\zeta)}{\zeta-z}\,\diff\zeta,
\]
so that $f\in\bigO(\Omega)$. Finally, $\st{\jmath(f)}(z)=\st{\jmath(f)(z)}=\left[\lim_{\eps\to0}f_{\eps}(z)\right]=f(z)$
and hence, if $\jmath(f)(z)\in\CC$, then $\st{\jmath(f)}(z)=\jmath(f)(z)=f(z)$.
\end{proof}
\begin{rem}
~
\begin{enumerate}
\item Following the definition of \cite{Eva}, the previous Thm.~\ref{thm:emb2}
allows us to include a subset of Sobolev spaces as GHF.
\item It is important to note that $\jmath(f)(z)=[f_{\eps}(z_{\eps})]$
is very different from its standard part, exactly as $\delta|_{\R}=0$
whereas $\delta(0)$ is an infinite number. The embedded GHF $\jmath(f)$
keeps all the smooth non-Archimedean information gained with the regularization
$f_{\eps}=f*k_{\eps}$. This is also confirmed by property \ref{enu:noNonArch}
of Thm.~\ref{thm:emb2}: if $\jmath(f)|_{\Omega}$ does not have
any non-Archimedean information, then the smoothing process is useless
and $\jmath(f)$ is only an extension of $f$ to generalized numbers
of $d_{f}\subseteq\RCcomplexrho$, i.e.~it equals the canonical embedding
\eqref{eq:canonical}.
\item Clearly, Thm.~\ref{thm:embeddingdistribution} allows us to also
embed bump functions or a nowhere analytic smooth function multiplied
by a bump which constantly equals $1$ on a neighborhood of a given
$\Omega$. In the former case, in \cite{NuGi24b} we will show that
any Taylor hyperseries of a GHF \emph{always} has a strictly positive
radius of convergence $r\in\RCrealrho_{>0}$. In case the center of
the hyperseries is a flat point, then the radius $r\approx0$ is infinitesimal:
indeed, a bump function identically equals $0$ (in $\RCcomplexrho$)
infinitely close to flat points. In the latter case, the regularization
process with holomorphic functions allows one to overcome the limitation
given by estimates \ref{enu:epsDerAreModerate} of Thm.~\ref{thm:epsDiff},
which does not hold for nowhere analytic smooth functions.
\end{enumerate}
\end{rem}

Finally, as we will explore in the subsequent paper of this series
(\cite{NuGi24a}), the path integral can be used to define a primitive
of generalized holomorphic functions, which opens up the possibility
of embedding a broader class of distributions, e.g.~a Heaviside like
function.

\section{\protect\label{subsec:Comparison-with-Colombeau}Comparison with
Colombeau theory and conclusions}

In Sec.~\ref{subsec:epsHol}, by comparing Thm.~\ref{thm:epsDiff}
with \cite[Thm.~2]{Oberguggenberger2007}, we already showed that
Colombeau generalized holomorphic functions are all included as GHF.
The family of examples given in the previous section shows that this
inclusion in meaningfully strict. See \cite{DHPV,Delcroix2005} for
the relations about how to include periodic hyperfunctions in Colombeau
type algebras. In this closing section, we want to underscore some
analogies and differences between our construction and Colombeau theory.

First of all, in both theories, Schwartz impossibility theorem about
multiplication of distributions (see e.g.~\cite{Grosser2001} and
references therein) is bypassed because only the pointwise product
of ordinary holomorphic functions is preserved, not that of continuous
functions. Using Thm.~\ref{thm:embeddingdistribution}, in general
we only have $\iota(f)\cdot\iota(g)\approx\iota(f\cdot g)$ if $f$,
$g$ are continuous and compactly supported, where $\approx$ has
been defined in Def.~\ref{def:nonArchNumbs}. This is clearly very
natural if one thinks at the regularization process of any embedding
and hence the related augmented non-Archimedean information.

A direct comparison between generally recognized technical drawbacks
of Colombeau theory and the theory of GHF are as follows:
\begin{enumerate}
\item Colombeau generalized functions are not closed with respect to composition,
and only a sufficient condition is possible (see \cite[Def.~1.2.7]{Grosser2001}),
whereas GHF are closed with respect to composition, see Thm.~\ref{thm:chain};
\item In general, Colombeau generalized functions can be defined only on
finite points of $\sint{\Omega}_{\mathrm{fin}}$, whereas GHF can
be defined on more general domains, see Def.~\ref{def:complexdiff}.
This implies that we can consider solutions of differential equations
e.g.~defined on infinitesimal interval and that cannot be extended
or defined in domains containing infinite numbers. For example, even
in the real case of GSF, consider $y'=-\frac{t}{1+y}\cdot\frac{1}{h}$,
$y(0)=0$, where $h$ is a positive invertible infinitesimal, and
whose real solution is $y(t)=-1+\sqrt{1-\frac{t^{2}}{h}}$, $t\in(-\sqrt{h},\sqrt{h})$.
This limitation of Colombeau theory is essentially due to the bias
of defining a sheaf of generalized functions starting from open sets
$\Omega\subseteq\R^{n}$ in order to have a more direct comparison
with classical distribution theory. In our construction, we started
assuming a non-Archimedean point of view by changing the ring of scalars.
\item Colombeau theory naturally assumes the possibility to take arbitrary
derivatives and requires, by definition, that generalized holomorphic
functions satisfy the CRE. On the contrary, we made an effort of asking
only a first order condition, so that CRE, Goursat, Looman-Menchof
and Montel theorems of Sec.~\ref{subsec:The-Cauchy-Riemann-equations}
can be proved. Similarly, a more flexible criterion of holomorphicity
is also provable, Thm.~\ref{thm:epsDiff}.
\item Colombeau theory of generalized holomorphic functions, see e.g.~\cite{Vernaeve2008},
uses ordinary series $\sum_{n\in\N}\frac{f^{(n)}(z_{0})}{n!}(z-z_{0})^{n}$
and hence, as we already explained in Sec.~\ref{subsec:The-Ring-ofRC},
good convergence result are not possible. The use of hyperseries $\hypersum{\rho}{\sigma}\frac{f^{(n)}(z_{0})}{n!}(z-z_{0})^{n}$,
i.e.~of series extended over \emph{infinite} generalized hypernatural
numbers, in \cite{NuGi24b} solves this type of problems.
\item For the sake of completeness, we only mention here that the notion
of Fourier transform of Colombeau generalized functions presents several
limitations in important theorems, such as the inversion one. On the
contrary, with the notion of hyperfinite Fourier transform, we can
consider the Fourier transform of any GSF, without restriction to
tempered type, and these limitations are removed, see \cite{MTG}.
\end{enumerate}
Concerning the intrinsic embedding of distributions, which is frequently
presented as a drawback of Colombeau theory, first of all we have
to explicitly state that this objection is wrong, because it is well-known
that considering a suitable index set instead of $I=(0,1]$, allows
one to have an intrinsic embedding, see \cite{Grosser2001}.\medskip{}

On the other hand, in the present paper we tried to underscore a different
point of view: at least since the work of Leray \cite{Ler}, smoothing
possible singularities using the dependence from a suitable parameter
$\eps$ is now a natural and frequently used method (see e.g.~\cite{BLL,Bet,MonT,BJ,LPPZ}
for recent works using similar methods). Therefore, in our opinion,
the present theory of GHF has to be considered as the introduction
of a mathematical structure of the form $(\rho,\ghf(-),\mu,k)$, where
$\mu$ is any entire mollifier for the embedding of compactly supported
distributions, and $k$ is a compactly supported function with integral
$1$ for the embedding of locally integrable functions satisfying
the distributional CRE. The free variables $\rho$, $\mu$ and $k$
need to be chosen depending on the regularization problem we have
to solve. Searching for hypothetical ``best or intrinsic'' values
of these variables would imply having less freedom in this method.\medskip{}

\noindent However, it is interesting to note that this freedom is
already necessarily constrained by the following uniqueness result
of the embedding of Thm\@.~\ref{thm:embeddingdistribution}:
\begin{thm}
\label{thm:uniquenessMoll}Let $\mu_{1}$, $\mu_{2}\in\Coo(\R^{2})$
be entire mollifiers and $\frac{\delta_{1}^{(n)}(0)}{n!}=[\delta_{1n,\eps}]$,
$\frac{\delta_{2}^{(n)}(0)}{n!}=[\delta_{2n,\eps}]$ be the coefficients
of the Taylor formula at $0$ of the corresponding Dirac delta GHF,
i.e.~$\delta_{1}(z)=\diff\rho^{-2}\mu_{1}(\diff\rho^{-1}z)$ and
$\delta_{2}(z)=\diff\rho^{-2}\mu_{2}(\diff\rho^{-1}z)$ respectively.
If these coefficients are strongly $\rho$-equivalent, denoted by
$(\delta_{1n,\eps})\simeq_{\rho}(\delta_{2n,\eps})$ as in \cite[Def.~3.(ii)]{Tiwari2023},
i.e. 
\begin{equation}
\forall q,r\in\N\,\forall^{0}\eps\,\forall n\in\N:|\delta_{1n,\eps}-\delta_{2n,\eps}|\leq\rho_{\eps}^{nq+r},\label{eq:strongrhoequivalent}
\end{equation}
then $\mu_{1}=\mu_{2}$. In other words, if the embedding $\iota(-)$
of Thm.~\ref{thm:embeddingdistribution} preserves the Taylor formula
of Dirac delta, then it does not depend on the entire mollifier $\mu$.
\end{thm}

\noindent This theorem will be fully clear only after the introduction,
in \cite{NuGi24b}, of the notion of equality between hyperseries,
which is given exactly by \eqref{eq:strongrhoequivalent}, (see also
\cite{Tiwari2023} for the case of generalized real analytic functions).
However, it can already be glimpsed looking at \ref{enu:epsDerAreModerate}
of Thm.~\ref{thm:epsDiff}: the upper bound
\[
\exists Q,R\in\R_{>0}\,\forall^{0}\eps\,\forall n\in\N:\ \left|\frac{f_{\eps}^{(n)}(z_{\eps})}{n!}\right|\le\rho_{\eps}^{-nQ-R}
\]
is analogous to the moderateness condition in Def.~\ref{def:RCring},
and hence \eqref{eq:strongrhoequivalent} corresponds to the related
negligibility condition.
\begin{proof}[Proof of Thm.~\ref{thm:uniquenessMoll}]
For all $n\in\N$, we have
\[
\delta_{1n,\eps}=\frac{\delta_{1\eps}^{(n)}(0)}{n!}=\frac{\rho_{\eps}^{-n-2}\mu_{1}^{(n)}(0)}{n!}\;\text{and}\;\delta_{2n,\eps}=\frac{\delta_{2\eps}^{(n)}(0)}{n!}=\frac{\rho_{\eps}^{-n-2}\mu_{2}^{(n)}(0)}{n!}.
\]
Setting $q=1$ and $r=0$ in \eqref{eq:strongrhoequivalent}, for
$\eps$ small and for all $n\in\N$, we have 
\[
\left|\frac{\mu_{1}^{(n)}(0)}{n!}-\frac{\mu_{2}^{(n)}(0)}{n!}\right|\leq\rho_{\eps}^{nq-n+2+r}\leq\rho_{\eps}^{2}.
\]
Taking $\eps\to0$, we can conclude that $\frac{\mu_{1}^{(n)}(0)}{n!}=\frac{\mu_{2}^{(n)}(0)}{n!}$,
for all $n\in\N$. Then $\mu_{1}=\mu_{2}$ since $\mu_{1}$ and $\mu_{2}$
are entire functions.
\end{proof}
\noindent Stating this uniqueness result in different words: if we
need to chose the entire mollifier $\mu$ in order to have specific
smoothing property, this corresponds to chose a particular form of
the Taylor formula of the corresponding Dirac delta. Therefore, the
freedom in choosing $\mu$ corresponds to the freedom of selecting
a matched Dirac delta embedded as a GHF. This is also natural if we
think at the infinite amount of non-Archimedean properties that $\delta(-)\in\ghf(\RCcomplexrho)$
satisfies whilst the classical ``macroscopic'' version $(\phi\in\mathcal{D}(\R^{2})\mapsto\phi(0)\in\R)\in\mathcal{D}'(\R^{2})$
does not.\medskip{}

For the sake of clarity, we close this section showing that the classical
ideas to prove the identity theorem for holomorphic function can also
be repeated for GHF. However, the identity principle does not hold
in our framework (see also \cite[Thm. 39 and 40]{Tiwari2023} for
the case of generalized real analytic functions) exactly because we
are in a non-Archimedean setting: e.g.~every interval of $\RCrealrho$
is not connected in the sharp topology because the set of all the
infinitesimals is a clopen set, see e.g.~\cite{Giordano2013}. Therefore,
repeating the classical proof of the identity theorem leads us to
state that the set of points where two given GHF agree is ``only''
a clopen set.

In \cite{NuGi24b} we will be able to prove the following natural
\begin{lem}
\label{lem:identitythm}Let $f$, $g\in\ghf(U,\RCcomplexrho)$ and
$c\in U$. If for all $n\in\N$ we have $f^{(n)}(c)=g^{(n)}(c)$ then
\[
\exists r\in\RCrealrho_{>0}:f\mid_{B_{r}(c)}=g\mid_{B_{r}(c)}.
\]
\end{lem}

\noindent We can also think at the following theorem as a result that
does not depend on \cite{NuGi24b}, if we take the property stated
in this Lemma \ref{lem:identitythm} as a natural assumption.
\begin{thm}
\label{thm:identitythm}Let $U\subseteq\RCcomplexrho$ be an open
set and $f$, $g\in\ghf(U,\RCcomplexrho)$. Then the set 
\[
\mathcal{O}:=\emph{int}\left\{ z\in U\mid f(z)=g(x)\right\} 
\]
is clopen in the sharp topology.
\end{thm}

\begin{proof}
For simplicity, considering $f-g$, without loss of generality we
can assume $g=0$. We only have to show that the set $\mathcal{O}$
is closed in $U$. Assume that $c$ is in the closure of $\mathcal{O}$
in $U$, i.e.
\begin{equation}
c\in U\:\forall r\in\RCrealrho_{>0}\:\exists\bar{c}\in B_{r}(c)\cap\mathcal{O}.\label{eq:cisintheclosure}
\end{equation}
We have to prove that $c\in\mathcal{O}.$ We first note that for each
$\bar{c}\in\mathcal{O}$, we have $B_{p}(\bar{c})\subseteq\mathcal{O}$
for some $p\in\RCrealrho_{>0}$ and hence
\begin{equation}
f(\bar{z})=0\quad\forall\bar{z}\in B_{p}(\bar{c}).\label{eq:defO}
\end{equation}
Now, fix $n\in\N$ in order to prove that $f^{(n)}(c)=0$. From \eqref{eq:cisintheclosure},
for all $r\in\RCrealrho_{>0}$ we can find $\bar{c}_{r}\in B_{r}(c)\cap\mathcal{O}$
such that $f^{(n)}(\bar{c}_{r})=0$ from \eqref{eq:defO}. From sharp
continuity of $f^{(n)}$, we have $f^{(n)}(c)=\lim_{r\rightarrow0^{+}}f^{(n)}(\bar{c}_{r})=0$
for all $n\in\N$. Therefore, since $f\in\ghf(U,\RCcomplexrho)$ and
$c\in U$, by Lemma \ref{lem:identitythm}, we can hence find $\delta>0$
such that $f(z)=0$ for all $z\in B_{\delta}(c)$, i.e.~$c\in\mathcal{O}$.
\end{proof}
The lacking of a general identity theorem is hence a feature of GHF
because it enables the inclusion of a wide range of interesting generalized
functions, paving also the way for a more general version of the Cauchy-Kowalevski
theorem. Since functions in Sobolev spaces can be approximated using
Taylor series, see \cite{Boj,CaCi,Spe}, and GHF solutions of PDE
equals their Taylor hyperseries, we also hope to find a way to understand
when a GHF solution is actually the embedding of a function in a Sobolev
space.

\section*{Acknowledgements}

Financial support for an Ernst-Mach PhD grant (Reference number =
MPC-2021-00491) issued by the Federal Ministry of Education, Science
and Research (BMBWF) through the Austrian Agency for Education and
Internationalization (OeAD-Gmbh) within the framework ASEA-UNINET
(https://asea-uninet.org/) for S. Nugraheni is gratefully acknowledged.
This research was also funded in whole or in part by the Austrian
Science Fund (FWF) P34113, 10.55776/P33945, 10.55776/P33538. For open
access purposes, the author has applied a CC BY public copyright license
to any author-accepted manuscript version arising from this submission.

\end{document}